\documentclass{article}
\usepackage{amsmath, amssymb, amsthm, latexsym, mathtools, url}
\usepackage{tikz}
\usetikzlibrary{calc}

\title{Topological expanders, coarse geometry and thick embeddings of complexes}
\author{David Hume, University of Birmingham}
\date{\today}

\newtheorem{proposition}{Proposition}[section]
\newtheorem{theorem}[proposition]{Theorem}
\newtheorem{corollary}[proposition]{Corollary}

\newtheorem{lemma}[proposition]{Lemma}
\newtheorem{fact}[proposition]{Fact}

\newtheorem*{theorem*}{Theorem}
\newtheorem*{corollary*}{Corollary}
\theoremstyle{definition}
\newtheorem{definition}[proposition]{Definition}
\newtheorem{remark}[proposition]{Remark}
\newtheorem{question}[proposition]{Question}
\newtheorem{guess}[proposition]{Guess}

\newcommand{\setcon}[2]{\left\{#1\ \left|\ #2\right.\right\}}
\newcommand{\varsetcon}[2]{\left\{#1\ :\ #2\right\}}
\newcommand{\abs}[1]{\left\lvert#1\right\rvert}

\newcommand{\R}{\mathbb{R}}
\newcommand{\Z}{\mathbb{Z}}
\newcommand{\N}{\mathbb{N}}
\newcommand{\HH}{\mathbb{H}}
\newcommand{\SOL}{\tu{SOL}}
\newcommand{\SL}{\tu{SL}}

\newcommand{\tu}{\textup}

\newcommand{\sep}{\tu{sep}}

\newcommand{\supp}{\tu{supp}}
\newcommand{\asdim}{\tu{asdim}}
\newcommand{\card}{\tu{card}}
\newcommand{\im}{\tu{im}}
\newcommand{\cut}{\tu{cut}}
\newcommand{\diam}{\tu{diam}}
\newcommand{\vol}{\tu{vol}}
\newcommand{\con}{\tu{con}}

\newcommand{\Cov}{\tu{Cov}}
\newcommand{\sOv}{{}_s\tu{Ov}}
\newcommand{\Ov}{\tu{Ov}}

\newcommand{\sTO}{{}_s\tu{TO}}
\newcommand{\TO}{\tu{TO}}
\newcommand{\cOv}{{}_c\tu{Ov}}
\newcommand{\cTO}{{}_c\tu{TO}}
\newcommand{\tS}{\overline{S}}
\newcommand{\bS}{\mathbb{S}_b}
\newcommand{\fS}{\mathbb{S}_f}
\newcommand{\CX}{\mathcal{C}X}
\newcommand{\CY}{\mathcal{C}Y}
\newcommand{\CZ}{\mathcal{C}Z}
\newcommand{\cw}{\tu{cw}}

\begin{document}
\maketitle
\begin{abstract} We quantify the topological expansion properties of bounded degree simplicial complexes in terms of a family of sublinear functions, in analogy with the separation profile of Benjamini-Schramm-Tim\'ar for classical expansion of bounded degree graphs. We prove that, like the separation profile, these new invariants are monotone under regular maps between complexes satisfying appropriate higher connectivity assumptions. In the dimension $1$ case, we recover the cutwidth profile of Huang-Hume-Kelly-Lam. We also prove the seemingly new result that any $1$-dimensional topological expander necessarily contains a graphical expander. 

In higher dimensions, we give full calculations of these new invariants for Euclidean spaces, which are natural analogues of waist and width-volume inequalities due to Gromov and Guth respectively. We present several other methods of obtaining upper bounds including na\"ive (yet useful) direct product and fibring theorems, and show how lower bounds can be obtained via thick embeddings of complexes, in analogy with previous work of Barrett-Hume. Using this, we find lower bounds for $k$-expansion of $(k+1)$-fold horocyclic products of trees, and for rank $k$ symmetric spaces of non-compact type. As a further application, we prove that for every $k\geq 2$ there is no coarse embedding (and more generally, no regular map) from the $k$-fold horocyclic product of $3$-regular trees to either any product $(\HH^2)^{k-2}\times H \times D$ where $\HH^2$ is the real hyperbolic plane, $H$ is a bounded degree hyperbolic graph and $D$ is a doubling metric space, or to any symmetric space whose non-compact factor has corank (dimension minus rank) is strictly less than $k$.
\end{abstract}

\section{Introduction}
During the last 50 years, the study of expander graphs has become an important part of both theoretical computer science and several areas of pure mathematics, linking them together in fruitful and unexpected ways, see for instance \cite{HLW-exp,Lub-exp} and the references therein. More recently, there has been intensive work on generalising the study of graph expansion to the expansion of higher-dimensional simplicial complexes. For graphs, there are several definitions of expansion coming naturally from different mathematical perspectives, and for bounded-degree graphs, many of these turn out to be equivalent up to constant factors. However, the natural generalisations of these notions to simplicial complexes are not equivalent; they lead to genuinely different notions of expansion, each with their own importance and applications, for example in property testing, locally testable codes, quantum computation, probabilistically checkable proofs, and the study of Kazhdan's Property $(T)$ for broad classes of infinite groups \cite{Breuckmann-Eberhardt,DELLM,Lub-HDE}.

Classical graphical expansion has been quantified in coarse geometric invariants such as the separation and Poincar\'e profiles. These profiles exhibit tantalising connections to other important invariants in geometric group theory (e.g., growth, isoperimetry and conformal dimension), and can be used to generalise some classical dichotomy results for connected unimodular Lie groups \cite{BenSchTim-12-separation-graphs,LeCozGour,HumeMackTess-Pprof,HumeMackTess-PprofLie}. The goal of this paper is to extend this theory and its connections to high dimension expanders.

The notion of high dimension expansion we will focus on here is topological expansion, due to Gromov \cite{Gro-10-exp-top}. Roughly speaking, an infinite family of finite $d$-dimensional simplicial complexes $(S_n)_{n\in\N}$ is a \textbf{$d$-dimensional topological expander} if there exists a constant $c>0$, such that for every $n$ and every continuous map $f:S_n\to\R^d$ there is a point whose preimage intersects at least a $c$-proportion of all the $d$-simplices in $S_n$. Gromov proved that, for each $d$, the family $(S_n)$ of $d$-skeleta of $n$-simplices is a $d$-dimensional topological expander, and asked whether there was an example with uniformly bounded degree. Evra-Kaufman proved that there are using the $d$-skeleta of a family of $(d+1)$-dimensional (bounded degree) Ramanujan complexes, extending previous work of Kaufman-Kazhdan-Lubotzky \cite{Gro-10-exp-top,KaufKazLub,EvraKaufman}.

In this paper we construct invariants which quantify for a given (infinite) bounded degree simplicial complex $X$ how close to $q$-dimensional topological expanders the finite subcomplexes of $X$ can be. Specifically, $\sTO^q_X(r)\leq k$ if every subcomplex $Z$ of $X$ with $r$ $0$-simplices admits a continuous map to $\R^q$ such that the preimage of every point intersects at most $k$ simplices of $Z$. The resulting invariants $\sTO^q_X$ are a family of sublinear functions which satisfy the following properties:

\begin{theorem}\label{thm:TOprops} Let $X$ and $Y$ be bounded degree simplicial complexes of dimension at most $d$.
\begin{enumerate}
 \item[(i)] $\liminf_{r\to \infty}\frac1r\sTO^q_X(r) \neq 0$ if and only if $X$ contains a $q$-dimensional topological expander;
 \item[(ii)] for every $q\geq 1$, $\sTO^{q+1}_X(r) \leq \sTO^q_X(r)$;
 \item[(iii)] for every $q>d$, $\sTO^q_X(r)$ is bounded.
 % \item[(iv)] for every $q,q'\geq 1$, $\sTO^{q+q'}_{X\times Y}(r) \leq \sTO^q_X(r)\sTO^{q'}_Y(r)$.
\end{enumerate}
\end{theorem}

Theorem \ref{thm:TOprops}(i) is an exact analogue of \cite[Theorem $1.3$]{HumSepExp} for usual graph expanders and the separation profile, while Theorem \ref{thm:TOprops}(iii) shows that bounded $\sTO^q_X$ is a much weaker condition than the existence of a topological embedding $X\to\R^q$. \medskip

Our goal is to use topological overlap profiles as coarse geometric tools. This is accomplished via the following monotonicity result. Throughout the paper, given two functions $f,g:\N\to\N$ we write $f\lesssim g$ if there is a constant $C$ such that $f(r)\leq Cg(Cr)+C$ for all $r\in\N$ and we write $f\simeq g$ if $f\lesssim g$ and $g\lesssim f$.

\begin{theorem}\label{thm:TOmono} Let $X,Y$ be bounded degree simplicial complexes, where $X$ has dimension $d$ and $Y$ is \textbf{uniformly $k$-connected} for all $k\leq d$. If there is a regular map $X^0\to Y^0$, then for all $q\geq 1$
\[
 \sTO^q_X(r)\lesssim \sTO^q_Y(r).
\]
\end{theorem}
We denote the set of $0$-simplices in $X$ by $X^0$, and equip this set with the shortest path metric in the $1$-skeleton; a map $X^0\to Y^0$ is \textbf{regular} if it is Lipschitz and pre-images of points have uniformly bounded cardinality\footnote{For maps between bounded degree graphs, this definition of regular map is equivalent to the definition presented in \cite{BenSchTim-12-separation-graphs} and elsewhere.}; and a metric space is \textbf{uniformly $k$-connected} if, for every $s$ there is some $S$ such that whenever there is a continuous map $f:S^{k-1}\to X$ whose image has diameter at most $s$, there is a continuous extension $g:D^k\to X$ of $f$ where the image of $g$ has diameter at most $S$\footnote{Simplicial complexes admit a natural metric which extends the shortest path metric on the $1$-skeleton (cf.\ \ref{defn:topreal})}. This uniformity property is required to ensure that there is a continuous extension $X\to Y$ of the regular map $X^0\to Y^0$ such that the image of each simplex is contained in a uniformly bounded neighbourhood of the image of its $0$-skeleton.

\subsection{Topological overlap of arbitrary metric spaces}
We introduce a second definition of topological overlap profiles which is defined for metric spaces with \textbf{bounded geometry} (that is, for every $R\geq r>0$, there is some constant $C(R,r)$ such that every ball of radius $R$ is contained in a union of $C$ balls of radius $r$). In what follows we will assume the spaces under discussion have bounded geometry. We write $\TO^q_X(r)\leq k$ if, for every subset $Z\subset X$ which is contained in a union of $r$ balls of radius $1$ in $X$, there is a continuous map $f:Z\to\R^q$ such that the preimage of every point is contained in a union of $k$ balls of radius $1$ in $X$. This definition agrees with the simplicial definition on simplicial complexes equipped with a natural metric (Theorem \ref{thm:defnequiv}) and has its own monotonicity result (Proposition \ref{prop:cTOmonotone}). We will use both versions throughout the paper.

\subsection{Topological overlap of groups}
One motivation for studying obstructions to regular maps is that given two Cayley graphs $X,Y$ of finitely generated groups $G,H$ with respect to finite symmetric generating sets, there is a regular map $VX\to VY$ whenever some finite index subgroup $G'$ of $G$ is isomorphic to a subgroup of $H$. We therefore view regular maps as a coarse generalisation of both subcomplex and subgroup inclusion.

To use higher topological overlap profiles as obstructions to subgroups, it is no longer sufficient to work with Cayley graphs. For a group $G$ of type F (those that admit a finite classifying space), a natural choice of simplicial complex to associate to $G$ is the universal cover of some iterated barycentric subdivision of the classifying space. For groups which are not of type $F$, the following proposition shows that $\TO^1$ and $\TO^2$ depend only on the $1$- and $2$- skeleton of the complex.

\begin{proposition}\label{prop:12skeleton} Let $X$ be a bounded degree simplicial complex. For $q=1,2$, we have 
\[
	\sTO^q_{X^{\leq q}}(r)\simeq \sTO^q_{X}(r).
\] 
\end{proposition}
As a result, for a finitely generated group $G$ we may always calculate $\sTO^1_G$ using a Cayley graph, and for a finitely presented group $H$ we may calculate $\sTO^2_H$ using a barycentric subdivision of a Cayley $2$-complex. We do not know whether Proposition \ref{prop:12skeleton} holds for $q\geq 3$.

\subsection{Calculations in dimension $1$: $\TO^1$, cutwidth and separation}

We can quickly deduce many facts about $\sTO^1$ by proving it is equal (up to $\simeq$) to the cutwidth profile defined in \cite{HHKL23}. We recall that the cutwidth profile of a bounded degree graph $Y$ is the function
\[
 \cw_Y(r) = \max\setcon{\cw(\Gamma)}{\Gamma\leq Y,\ |\Gamma|\leq r}
\]
where $\cw$ denotes the cutwidth (cf.\ Definition \ref{defn:cutwidth}), $\Gamma\leq Y$ indicates that $\Gamma$ is a subcomplex (in this case subgraph) of $Y$, and $|\Gamma|$ is the number of $0$-simplices in $\Gamma$.

\begin{theorem}\label{thm:compTO1cutwidth} Let $X$ be a bounded degree simplicial complex, and let $X^{\leq 1}$ denote the $1$-skeleton of $X$. We have
\[    
    \sTO^1_X(r)\simeq \cw_{X^{\leq 1}}(r).
\]
\end{theorem}
Cutwidth profiles are calculated for many interesting classes of bounded degree graphs by comparing the cutwidth profile to the separation profile \cite[Theorem 1.4]{HHKL23} (cf.\ \ref{sec:TO1conseq}).

Using the same family of ideas, we prove that 1-dimensional topological expanders necessarily contain graphical expanders. The full statement with explicit constants is given in Corollary \ref{cor:topexp-graphexp2}.

\begin{corollary}\label{cor:topexp-graphexp}
Let $Z=\bigsqcup_n Z_n$, where $(Z_n)_{n\in\N}$ is a $1$-dimensional bounded degree topological expander. For each $n$ there is a subcomplex $\Gamma_n\leq Z_n$ such that $(\Gamma_n)_{n\in\N}$ is a graphical expander.
\end{corollary}

\subsection{Topologically thin and topologically thick spaces}
Inspired by the thick-thin dichotomy for Poincar\'e profiles introduced in \cite{HumeMackTess-PprofLie}, we make the following definition.

\begin{definition}
    Let $X$ be a metric space and let $q\in\N$. We say $X$ is $q$-\textbf{topologically thin} if $\TO^q_X(r)\lesssim r^a$ for some $a<1$ and we say $X$ is $q$-\textbf{topologically thick} if $\TO^q_X(r)\gtrsim r/\ln(1+r)^b$ for some $b\geq 0$.
\end{definition}
We note that all analytically thick (respectively thin) bounded degree graphs are $1$-topologically thick (resp.\ thin). Also, if $X$ is $q$-topologically thin, then it is $q'$-topologically thin for all $q'\geq q$; and if $X$ is $q$-topologically thick, then it is $q'$-topologically thick for all $1\leq q'\leq q$. Both these observations follow from Theorem \ref{thm:TOprops}(ii). Our next family of results will focus on upper and then lower bounds, with a particular emphasis on sufficient conditions to deduce topological thickness or thinness.

Combining Theorem \ref{thm:compTO1cutwidth} with results from \cite{HHKL23,HumeMackTess-Pprof,HumeMackTess-PprofLie} we deduce that 
\begin{itemize}
    \item $H\times D$ is $1$-topologically thin whenever $H$ is a hyperbolic, and $D$ satisfies the doubling property (for instance, a product of a rank one simple Lie group with a connected nilpotent Lie group),
    \item $X$ is $1$-topologically thick whenever there is a regular map of the horocyclic product of 2 $3$-regular trees into $X$ which has uniformly discrete image (this happens for unimodular solvable Lie groups of exponential growth, semisimple Lie groups of rank at least 2 and any Cayley graph of a group containing a direct product of two free groups as a subgroup).
\end{itemize}

We now move to finding bounds in higher dimensions. Results are presented below, opinions and guesses as to how sharp these bounds are appear in \S\ref{sec:guess}.

\subsection{Euclidean spaces}

Our first calculation in higher dimensions is a complete description of the topological overlap of Euclidean spaces.

\begin{theorem}\label{thm:TOEuc} For each $1\leq q< n$ we have
\[
 \TO^q_{\R^n}(r)\simeq r^{1-q/n}.
\]
For every $q\geq n$, $\TO^q_{\R^n}(r)\simeq 1$.
\end{theorem}
The lower bound (for $q<n$) can be deduced from one of Gromov's several proofs of the waist inequality (specifically, \cite[p134]{Gromov83}), while the proof for the upper bound is much more technical, and is similar in spirit to Guth's proof of the width-volume inequality \cite{Guth-widthvol}.

\subsection{Upper bounds in dimensions $2$ and higher}

We present two general tools for upper bounds: a product theorem and a fibring theorem. All metric spaces considered are assumed to have bounded geometry.

\begin{theorem}\label{thm:productub}
    Let $X,Y$ be metric spaces. For every $q,q'\geq 1$
    \[
        \TO^{q+q'}_{X\times Y}(r) \lesssim \TO^q_X(r)\TO^{q'}_Y(r)
    \]
\end{theorem}

We record two simple corollaries of this and known results for $\TO^1$.

\begin{corollary}\label{cor:prodH2Tub}
    Let $X=(\HH^2)^k\times (T_3)^l$, where $T_3$ is the infinite $3$-regular tree. We have
    \[
        \TO^{k+l}_X(r) \lesssim \left(\TO^1_{\HH^2}(r)\right)^k\left(\TO^1_{T_3}(r)\right)^l \lesssim \ln(1+r)^{2k+l}
    \]
\end{corollary}
In particular, $(\HH^2)^k\times (T_3)^l$ is $(k+l)$-topologically thin. Later we will prove that these spaces are $(k+l-1)$-topologically thick.
\begin{corollary}\label{cor:prodXR}
    Let $X$ be a metric space. For every $q,k\geq 1$
    \[
        \TO^{q+k}_{X\times\R^k}(r) \lesssim \TO^q_X(r)\TO^k_{\R^k}(r) \simeq \TO^q_X(r)
    \]
\end{corollary}
In Corollary \ref{cor:prodXR}, we can replace $\R^k$ by any metric space admitting a continuous injection into $\R^k$, or more generally any metric space $Y$ where $\TO^k_Y$ is bounded.

Note that for $H=Y=\HH^3$ and $q=q'=1$, Theorem \ref{thm:productub} gives the trivial upper bound of $\TO^2_{(\HH^3)^2}(r)\lesssim r$, which is worse than the (currently best) bound of $\TO^2_{(\HH^3)^2}(r)\lesssim \TO^1_{(\HH^3)^2}(r) \lesssim r/\ln(1+r)$. The best known lower bound is $\TO^2_{(\HH^3)^2}(r)\gtrsim \TO^2_{\R^4}(r)\simeq r^\frac12$, which follows from Theorem \ref{thm:TOEuc} and the monotonicity of $\TO$.

Next we present a fibring theorem, which takes further advantage of the metric definition of $\TO$.

\begin{proposition}\label{prop:semidirprod} Let $(A,d_A)$ be a metric space and let $(B,d_B)$ be a metric space which admits a continuous injective map $g:B\to\R^k$ for some $k$. Let $d$ be a metric on the set $A\times B$ which satisfies
\begin{equation*}
 d((a,b),(a',b))=d_A(a,a')
\end{equation*}
for all $a,a'\in A$ and all $b\in B$.

If there is some constant $C$ such that $\Cov^1(\pi_A(Z))\leq C\Cov^1(Z)$ for all $Z\subseteq A\times B$ with $\Cov^1(Z)<+\infty$, where $\pi_A$ denotes projection onto the $A$-factor,
 then for all $q\in\N$
\[
\TO^{k+q}_{(A\times B,d)}(r) \leq \TO^q_{(A,d_A)}(Cr).
\]
\end{proposition}

One application of this result is to obtain the following upper bounds for real hyperbolic spaces

\begin{theorem}\label{thm:TOhyp} We have
\[
    \TO^q_{\HH^d}(r) \lesssim  \left\{
                                    \begin{array}{cc}
                                        r^{1-1/(d-q)} & 2 \leq q \leq d-2 \\
                                        \ln(1+r)^2 & q=d-1
                                    \end{array}
                            \right.
\] 
\end{theorem}

Proposition \ref{prop:semidirprod} can be applied to short exact sequences of topological groups where the normal subgroup admits a continuous injective map into some $\R^k$. 

\begin{corollary}\label{cor:ses} Let $1\to N \to G \to_{\pi} H \to 1$ be a (continuous) short exact sequence of compactly generated locally compact groups, and let $d_G$ be a proper left-invariant metric on $G$. If there is a continuous injective map $N\to \R^k$ for some $k$, then for all $q\in\N$, we have
\[
\TO^{k+q}_{(G,d_G)}(r) \lesssim \TO^q_{(H,d_H)}(r)
\]
where $d_H$ is the proper left-invariant metric $d_H(h,h')=d_G(\pi^{-1}(h),\pi^{-1}(h'))$.
\end{corollary}

Two interesting consequences of this are

\begin{corollary} Let $1\to N\to G \to \R^l\to 1$ be a short exact sequence where $N$ is a connected simply-connected nilpotent Lie group of manifold dimension $k$. For every $q\geq 1$
\[
    \TO^{k+q}_G(r) \lesssim \TO^q_{\R^l}(r)\simeq r^{1- \frac{q}{l}}
\]
\end{corollary}
In particular, for the real Heisenberg group $H_m$ of dimension $2m+1$ and $q\geq 2$, $\TO^q_{H_m}(r)\lesssim r^{1-\frac{q-1}{2m}}$. The embedded copy of $\R^{m+1}$ in $H_m$ gives a lower bound $\TO^q_{H_m}(r)\gtrsim r^{1-\frac{q}{m+1}}$.

\begin{corollary}\label{cor:ubsymspace} Let $X$ be a Riemannian symmetric space with no compact factor of dimension $n$ and rank $d$. For every $q\geq 1$ we have
    \[
        \TO^{n-d+q}_X(r)\lesssim \TO^q_{\R^d}(r)\simeq r^{1-\frac{q}{d}}.
    \]
\end{corollary}
In particular, $X$ is $(n-d+1)$-topologically thin.

The final result in this section is an upper bound for all metric spaces with finite Assouad-Nagata dimension, which just follows from the equivalent result for $\TO^1$ and Theorem \ref{thm:TOprops}(ii).

\begin{corollary}\label{cor:TO1finANdim} Let $X$ have finite Assouad-Nagata dimension. Then, for every $q\geq 1$,
\[
 \TO^q_X(r) \lesssim \frac{r}{\ln(1+r)}.
\]
\end{corollary}

\subsection{Lower bounds in dimensions $2$ and higher}
As $\TO$ behaves monotonically with respect to regular maps, one obvious source of lower bounds are embedded Euclidean spaces. In addition to the lower bounds on Heisenberg groups listed above, we have the following result for real hyperbolic spaces

\begin{proposition}%\label{prop:hyplb} We have
\[
    \TO^q_{\HH^d}(r) \gtrsim  \left\{
                                    \begin{array}{cc}
                                        r^{1-\frac{q}{d-1}} & 1 \leq q \leq d-2 \\
                                        \ln(1+r) & q=d-1
                                    \end{array}
                            \right.
\]
\end{proposition}
We note that these results are sharp when $d\geq 3$ and $q=1$.

A much more substantial result is the following sufficient condition to be $q$-topologically thick.
\begin{theorem}\label{thm:TOhorocycle}
    Let $H^{d+1}$ be the horocyclic product of $(d+1)$ $3$-regular trees (which is a $d$-dimensional simplicial complex). For every $1\leq q \leq d$, we have
    \[
        \sTO^q_{H^{d+1}}(r) \gtrsim r/ln(1+r)^q.
    \]
\end{theorem}
As all rank $d$ symmetric spaces of non-compact type and all thick Euclidean buildings of rank $d$ with co-compact affine Weyl group admit biLipschitz embeddings of $T_3^d$ \cite{BensaidNguyen}, we deduce that all these spaces are topologically $(d-1)$-thick.

Combining these calculations, we can prove the following obstruction to regular maps, which we believe is new even in the case of quasi-isometric embeddings.

\begin{theorem}\label{thm:regobst} There is no regular map from the $0$-skeleton of $H^{d+1}$ to
\begin{itemize}
    \item any product $H\times (\HH^2)^{d-1}\times D$ where $H$ is hyperbolic and $D$ satisfies the doubling property;
    \item any Riemannian symmetric space with corank (dimension minus rank) less than $d$.
\end{itemize}
\end{theorem}

\subsection{Coarse constructions and thick embeddings of simplicial complexes}
The strategy behind the proof of Theorem \ref{thm:TOhorocycle} is to prove high dimension analogues of two of the main technical results of \cite{BarrettHume}. We say that a continuous map between simplicial complexes $f: X\to Y$ is a \textbf{coarse construction} if, for every $\ell$, the image of $X^{\leq \ell}$ is contained in $Y^{\leq \ell}$. A coarse construction is called a \textbf{$k$-coarse construction} if the preimage of every simplex in $Y$ is contained in a union of at most $k$ simplices in $X$. The \textbf{volume} of a coarse construction is the minimal number of $0$-simplices in a subcomplex of $Y$ which contains the image of $f$.  

Natural analogues of \cite[Proposition 1.10]{BarrettHume} and \cite[Theorem 1.11]{BarrettHume} hold for coarse constructions and higher dimensional simplicial complexes, under additional connectedness hypotheses (as in Theorem \ref{thm:TOmono}). We will state but not prove these here as they are not our main focus. A topological embedding of a simplicial complex is said to be \textbf{$T$-thick} if the images of disjoint simplices are at distance at least $T$.

\begin{proposition}%\label{prop:coarse_wiring} 
Let $M$ be a Riemannian manifold with controlled growth,\footnote{For each $r>0$ the infimal and supremal volumes of $r$-balls in $M$ are positive and finite.} let $Y$ be a simplicial complex quasi-isometric to $M$ which is uniformly $q'$-connected for all $1\leq q'\leq q$, let $d\in\N$ and let $T>0$. There exist constants $C$ and $k$ such that for every finite simplicial complex $Z$ with dimension at most $q$ and maximal
    degree $d$ the following holds:

    If there is a $T$-thick topological embedding $Z\to M$ with volume $V$
    then there is a $k$-coarse construction of $Z$ into $Y$ with volume at most $CV$.
\end{proposition}

With stronger hypotheses we are able to convert coarse $k$-constructions into thick topological embeddings.

\begin{theorem}\label{thm:thickemb}
Let $M$ be a compact Riemannian manifold of dimension $n\geq 2q+1$, let $Y$ be a simplicial complex quasi-isometric to the universal cover $\widetilde{M}$ of $M$ which is uniformly $q'$-connected for all $1\leq q'\leq q$. Let $k,d\in\N$. There exist constants
    $C$ and $T>0$ such that the following holds:

    If there is a $k$-coarse construction of a finite simplicial complex $Z$ with dimension at most $q$ and maximal degree $d$ into $Y$ with volume $V$ then there is a $T$-thick embedding
    of $Z$ into $\widetilde{M}$ such that the volume of the $T$-neighbourhood of its image is at most $CV$.
\end{theorem}

Instead we focus on the relationship between the volume of $k$-coarse constructions and topological overlap profiles. The main result here is an analogue of \cite[Proposition 5.3]{BarrettHume}. Given two bounded degree simplicial complexes $X$ and $Y$, we write $\con^k_{X\to Y}(r)\leq m$ if every subcomplex $Z$ of $X$ with $r$ $0$-simplices admits a $k$-coarse construction into $Y$ with volume at most $m$.

\begin{theorem}\label{thm:TOcon} Let $X$ and $Y$ be bounded degree simplicial complexes and suppose there is some $k$ such that $\con^k_{X\to Y}(r)<+\infty$ for every $r$. Then for each $q\geq 1$,
\[
 \sTO^q_X(r)\leq k\sTO^q_Y(\con^k_{X\to Y}(r)).
\]
\end{theorem}

Combining this with the Evra-Kaufman construction of bounded degree topological expanders in every dimension, we obtain the follow variant of \cite[Theorem 1.17]{BarrettHume}

\begin{theorem}\label{thm:TOlbcon} Let $Y$ be a finite-dimensional bounded degree simplicial complex, and fix $q$. If, for every $d\in\N$ there is a constant $k=k(d,q)$ such that every finite simplicial complex $Z$ with dimension at most $q$ and degree at most $d$ admits a $k$-coarse construction into $Y$ with volume at most $k|Z|^a\ln(1+|Z|)^b$, then
\[
 \sTO^q_Y(r)\gtrsim r^{1/a}\ln(1+r)^{-b/a}.
\]
\end{theorem}

The proof of Theorem \ref{thm:TOhorocycle} is now an immediate consequence of Theorem \ref{thm:TOlbcon} and the following difficult generalisation of \cite[Theorem 1.18]{BarrettHume}:

\begin{theorem}\label{thm:conhorocycle} For every $q,d\in\N$ there is a $k=k(q,d)$ such that the following holds for every $r$. For every simplicial complex $Z$ with $r$ $0$-simplices, dimension at most $q$ and maximal degree at most $d$, there is a $k$-coarse construction of $Z$ into $H^{q+1}$ with volume at most $kr\ln(1+r)^q$.
\end{theorem}

\subsubsection{Coarse constructions of graphs into $\Z^2$ and applications to VLSI}
Minimizing the area of a circuit is an important problem in the domain of Very Large Scale Integration (VLSI). In the language of this paper, one of the key problems tackled in \cite{Le82} is to find upper bounds on $\con^k_{X\to \Z^2}(r)$ for interesting graphs $X$\footnote{Leiserson's results are stronger than this: they require the image to be contained within a rectangle of controlled eccentricity and are actually measuring the volume of the rectangle rather than the image as defined here}. One consequence of Leiserson's results is

\begin{theorem} Let $X$ be a bounded degree graph such that $\sep_X(r) \lesssim r^a\log(r)^b$. For all sufficiently large $k=k(\deg(X))$, we have
\begin{align*}
 \con^k_{X\to \Z^2}(n) & \lesssim 
 	\left\{ 
 		\begin{array}{lll}
 			n & \textup{if} & a<\frac12, \\
 			n\log(n)^{2b+2} & \textup{if} & a=\frac12, b\geq 0 \\
 			n^{2a}\log(n)^{2b} & \textup{if} & a>\frac12.
 		\end{array}
 	\right.
\end{align*}
\end{theorem}

As $\con^k_{X\to \Z^2}(n)\gtrsim n$ for any infinite connected graph $X$, the first of these bounds is optimal. Moreover, if $\sep_X(r) \simeq r^a\log(r)^b$ with $a>\frac12$, then $\cw_X(r) \simeq r^a\log(r)^b$ by \cite{HHKL23}, so by Theorem \ref{thm:TOcon} we have 
\[
    \con^k_{X\to \Z^2}(n) \simeq n^{2a}\log(n)^{2b}.
\]
The second bound is not optimal in all cases, this is obvious for $X=\Z^2$ and it is conjectured that $\con^k_{X\to \Z^2}(n)\simeq n$ holds for all planar graphs $X$ (with suitably large $k$ depending on $X$). A candidate graph with $\sep_X(r)\simeq r^{\frac12}$ which one may hope does not satisfy $\con^k_{X\to \Z^2}(n)\simeq n$ for any $k$ is when $X$ is an appropriate simplicial triangulation of real hyperbolic space of dimension 3.

\subsection{Questions and guesses}\label{sec:guess}
Here we collect some very natural open questions raised by the above discussion. The first question is whether $\sTO_X^q$ is determined by the $q$-skeleton of $X$.

\begin{question} Is there a bounded degree finite dimensional simplicial complex $X$ and a $q\in\N$ such that
\[
	\sTO^q_X(r)\not\simeq \sTO^q_{X^{\leq q}}?
\]
\end{question}
In proposition \ref{prop:12skeleton} we prove that no such complex exists when $q=1,2$. 

We also raise two natural questions concerning the relationship between topological thin/thickness and direct products.

\begin{question} If $X,Y$ are $q$- and $q'$-topologically thin respectively, is $X\times Y$ $(q+q')$-topologically thin? Conversely, if  $X,Y$ are $q$- and $q'$-topologically thick respectively, is $X\times Y$ $(q+q')$-topologically thick?
\end{question}
Note that the first question is still open for $X=Y=\HH^3$ and $q=q'=1$. Another potentially interesting direction is to consider whether the following variant of Corollary \ref{cor:ses} holds.

\begin{question}
    Let $1\to N \to G \to_{\pi} H \to 1$ be a (continuous) short exact sequence of compactly generated locally compact groups, and let $d_G$ be a proper left-invariant metric on $G$. Set $d_H(h,h')=d_G(\pi^{-1}(h),\pi^{-1}(h'))$. If there is a continuous injective map $(H,d_H)\to \R^k$ for some $k$, does it follow that, for all $q\in\N$, we have
\[
\TO^{k+q}_{(G,d_G)}(r) \lesssim \TO^q_{(N,d_G)}(r)?
\]
\end{question}
A positive resolution to this question would yield significant improvements to Theorem \ref{thm:TOhyp} and Corollary \ref{cor:ubsymspace}. In particular, we may make the following guess.

\begin{guess}
    Let $X$ be a Riemannian symmetric space of non-compact type or thick Euclidean building with co-compact affine Weyl group of rank $d$. Then $X$ is topologically $d$-thin.
\end{guess}
We recall that these spaces are topologically $(d-1)$-thick by Theorem \ref{thm:TOhorocycle} and \cite{BensaidNguyen}.
\medskip

Next, we list some consequences of the heuristic that the subsets which maximise $\TO^1$ may also maximise higher topological overlap profiles, and that the ``least overlapping'' continuous maps are ``obvious'' ones.

\begin{guess} We have
\[
\begin{array}{cc}
    \TO^q_{\HH^d}(r) \simeq \left\{ \begin{array}{cc}
        r^{1-q/(d-1)} & 1\leq q < d-1 \\
        \ln(1+r) & q=d-1
    \end{array}\right. 
\end{array}
\]
\end{guess}
Here we would be projecting metric balls to a totally geodesically embedded copy of $\HH^q$ inside $\HH^n$ meeting the centre of the ball (where $\HH^1$ would just be a biinfinite geodesic line). At the moment, we have no sensible guess for other symmetric spaces, however, since rank 1 symmetric spaces are topologically thin above the rank, and rank 2 symmetric spaces of non-compact type are $1$-topologically thick, we may hope for something like the following.

\begin{guess}\label{guess:symspace} Let $X$ be a symmetric space whose noncompact factor has rank $n$. For $1\leq q < n$, we have $\TO^q_X(r)\gtrsim r/\ln(1+r)^q$ and there is some $a=a(X)<1$ such that $\TO^n_X(r)\lesssim r^a$.
\end{guess}
We know that this holds when $n=1$, and that the lower bound for $\TO^1$ holds when $n\geq 2$. Since $\SL(n+1,\R)$ (which has rank $n$) contains a quasi-isometrically embedded copy of the direct product of $n$ copies of $\HH^2$ \cite{FisherWhyte} - and therefore contains the horocyclic product $H^{n}(T)$ - the lower bound below the rank also holds here.
\medskip

For the Heisenberg group, projecting a ball to the $y,z$-plane has geodesic fibres, so we make the following prediction.

\begin{guess} For the real Heisenberg group 
\[
H= \setcon{\left(\begin{array}{ccc}
    1 & x & z \\
    0 & 1 & y \\
    0 & 0 & 1
\end{array}\right)}{x,y,z\in\R}
\]
we have
\[
    \TO^2_H(r)\simeq r^{\frac14}.
\]    
\end{guess}

Finally, we make the following prediction for $\SOL$.

\begin{guess} We have $\TO^2_{\SOL}(r)\simeq \ln(1+r)$.
\end{guess}
We already know from Theorems \ref{thm:TOhyp} and \ref{thm:productub} that $\TO^2_{\SOL}(r)\lesssim \ln(1+r)^4$ since $\SOL$ quasi-isometrically embeds into $\HH^2\times \HH^2$. At the moment, we do not have a proof that either $\TO^2_H(r)$ or $\TO^2_{\SOL}(r)$ are unbounded.

\subsection{Plan of the paper}
In $\S\ref{sec:simp}$ we recall basic notation concerning simplicial complexes. The goal of $\S\ref{sec:defn}$ is to give the two definitions of topological overlap, show they are naturally equivalent and prove the key monotonicity result (Theorem \ref{thm:TOmono}). The dimension 1 case is considered in $\S\ref{sec:dim1}$, we prove that $\TO^1$ is equivalent to the cutwidth profile, and show that $1$-dimensional topological expanders necessarily contain graphical expanders. We move on to calculations in higher dimensions in $\S\ref{sec:highdim}$, firstly proving the product and fibring theorems and their consequences, then giving sharp bounds for Euclidean spaces. In $\S6$ we introduce coarse constructions, and prove Theorems \ref{thm:TOcon}, \ref{thm:TOlbcon} and \ref{thm:conhorocycle}.

\subsection{Acknowledgements}
The author would like to thank Goulnara Arzhantseva, David Ellis, Shai Evra, John Mackay, Romain Tessera and Federico Vigolo for enlightening conversations on a variety of topics in and around the area of high dimension expanders. This paper would not have been possible without the Focussed Research Workshop ``Higher-dimensional expanders and applications to coarse geometry'', funded by the Heilbronn Institute for Mathematical Research. The author was supported by the EPSRC Fellowship EP/V027360/1 ``Coarse geometry of groups and spaces''.

\section{Simplicial complexes}\label{sec:simp}

\begin{definition}\label{defn:simp} An \textbf{abstract simplicial complex} is a pair $(S,\mathcal S)$ where $S$ is a set and $\mathcal S$ is a collection of subsets of $S$ which contains all singletons and is closed under taking subsets. The \textbf{degree} of $(S,\mathcal S)$ is 
\[
	\deg(S,\mathcal S)=\max_{s\in S}\left|\setcon{\{s,t\}\in\mathcal S}{s\neq t}\right|.
\] 
The \textbf{dimension} of $(S,\mathcal S)$ is 
\[
	\dim(S,\mathcal S)=\max_{U\in\mathcal S}|U|-1.
\]
\end{definition}

Given two abstract simplicial complexes $(S,\mathcal S)$ and $(T,\mathcal T)$ we write $(S,\mathcal S)\leq (T,\mathcal T)$ if there is an injection $\iota:S\to T$ such that $\iota(U)\in\mathcal T$ for all $U\in\mathcal S$. We call $(S,\mathcal S)$ a \textbf{subcomplex} of $(T,\mathcal T)$.

We collect two useful facts.
\begin{fact}\label{fact:simplicebound}
\begin{itemize}
    \item If $|S|<+\infty$, then $|\mathcal S|\leq |S|2^{\deg(S,\mathcal S)}$.
    \item If $\deg(S,\mathcal S)<+\infty$, then
    \[
     \Delta_{(S,\mathcal S)} := \max_{\sigma\in \mathcal S}\setcon{\sigma'\in \mathcal S}{\sigma\cap\sigma'\neq\emptyset}<+\infty.
    \]
\end{itemize}
\end{fact}

\begin{definition}%\label{defn:flag}
An abstract simplicial complex $(S,\mathcal S)$ is called a \textbf{flag complex} if $U\in \mathcal S$ whenever every $2$ element subset of $U$ is in $\mathcal S$.
\end{definition}

A natural example of a flag complex is the barycentric subdivision of any abstract simplicial complex.

\begin{definition}\label{barycentre} Let $(S,\mathcal S)$ be an abstract simplicial complex. The \textbf{barycentric subdivision} of $(S,\mathcal S)$ is the pair $(\mathcal S,\mathbb S)$ where $\mathbb S$ is the set of increasing chains of elements of $\mathcal S$.
\end{definition}

We collect the following useful facts concerning barycentric subdivisions for later use.

\begin{fact}\label{fact:subdivprop} If $\dim(S,\mathcal S)<+\infty$, then $\dim(\mathcal S,\mathbb S)=\dim(S,\mathcal S)$. If, in addition, $\deg(S,\mathcal S)<+\infty$, then $\deg(\mathcal S,\mathbb S)\leq 2^{\deg(S,\mathcal S)}$.
\end{fact}
% Proof, it suffices to bound for each $\sigma\in\mathcal S$, the set of all $\tau\in\mathcal S$ which strictly contain $\sigma$. If $|\sigma|\geq 2$ then every $\tau$ containing $\sigma$ also contains every singleton subset of $\sigma$, so we may assume that $|\sigma|=1$. The number of $i$-cells in $\mathcal S$ is at most $\binom{\deg(S,\mathcal S)}{i-1}$, so $\deg(\mathcal S,\mathbb S)\leq \sum_{i=2}^{\dim(S,\mathcal S)}\binom{\deg(S,\mathcal S)}{i-1}\leq 2^{\deg(S,\mathcal S)}$.

\begin{definition}\label{defn:topreal}
The \textbf{topological realisation} of a simplicial complex $(S,\mathcal S)$ is the set
\[
 \tS=\setcon{f:S\to [0,1]}{\supp(f)\in\mathcal S,\ \sum_{s\in S}f(s)=1}
\]
equipped with the subspace topology from $[0,1]^S$. We may also consider $\tS$ as a metric space, where the distance between functions $f$ and $f'$ is given by
\begin{equation}\label{eq:scmetric}
	\min_k\varsetcon{\sum_{i=1}^k \|f_{i-1}-f_{i}\|_2}{f_i\in \tS,\ f_0=f,\ f_k=f',\ \supp(f_{i-1})\cup\supp(f_i)\in\mathcal S}.
\end{equation}
\end{definition}
The support function defines a natural surjective map $\supp:Z\to\mathcal S$.

\begin{remark}
    For the remainder of a paper we will use the phrase ``simplicial complex'' to mean the topological realisation of an abstract simplicial complex.
\end{remark}

\begin{definition} A \textbf{$k$-simplex} is an element of $\mathcal S$ with cardinality $k+1$. In the topological realisation an \textbf{open $k$-simplex} is the preimage of a $k$-simplex under $\supp$ and a \textbf{closed $k$-simplex} is the closure of an open $k$-simplex, or equivalently, the preimage under $\supp$ of all the subsets of an element of $\mathcal S$ with cardinality $k+1$.
\end{definition}
Given a simplicial complex $Z$, we denote the set of closed $k$-simplices by $\CZ^k$ and write $\CZ^{\leq k}=\bigcup_{i=1}^k\CZ^i$ and $\CZ=\bigcup_k \CZ^k$.

Given a simplicial complex $Z\subseteq \R^I$, we define
\[
 S_Z=\bigcup_{f\in Z} \supp(f) \quad \textup{and} \quad \mathcal S_Z = \setcon{\supp(f)}{f\in Z}.
\]
It is immediate that $Z$ is the topological realisation of the abstract simplicial complex $(S_Z,\mathcal S_Z)$. We write $\dim(Z)=\dim(S_Z,\mathcal S_Z)$, $\deg(Z)=\deg(S_Z,\mathcal S_Z)$ and $\Delta_Z=\Delta_{(S_Z,\mathcal S_Z)}$. We write $|Z|=|S_Z|$.

Given simplicial complexes $Y$ and $Z$, we say $Y$ is a subcomplex of $Z$ (and write $Y\leq Z$) if $(S_Y,\mathcal S_Y)$ is a subcomplex of $(S_Z,\mathcal S_Z)$.

\begin{remark} We denote by $\bS$ the set of all simplicial complexes $Z$ where $S_Z$ is countable and $\deg(Z)<+\infty$ (note that this immediately implies that $\dim(Z)<+\infty$), and by $\fS$ the set of all simplicial complexes $Z$ where $S_Z$ is finite.
\end{remark}

\begin{definition}\label{defn:kskel} Given a simplicial complex $(S,\mathcal S)$, we denote by $\mathcal S^k$ and $\mathcal S^{\leq k}$ the set of all elements of $\mathcal S$ with cardinality equal to/at most $k+1$ respectively. $(S,\mathcal S^{\leq k})$ is called the \textbf{$k$-skeleton} of $(S,\mathcal S)$.

The $k$-skeleton of $Z\in\bS$ is the set of all points in $Z$ whose support contains at most $k+1$ elements. Equivalently, if $Z$ is the topological realisation of $(S,\mathcal S)$ then $Z^{\leq k}$ is the topological realisation of $(S,\mathcal S^{\leq k})$.
\end{definition}

\begin{remark}\label{rem:scgraphs} We will sometimes use graph-theoretic terminology when discussing simplicial complexes of dimension $1$, referring to $0$-simplices as vertices and $1$-simplices as edges.
\end{remark}

\section{Topological overlap}\label{sec:defn}
The goal of this section is to give two definitions of topological overlap, one for simplicial complexes ($\S\ref{sec:defnsTO}$) and one for general metric spaces ($\S\ref{sec:covdim}$). We will then prove that they are equivalent on simplicial complexes equipped with the metric \eqref{eq:scmetric}. The main goal of the section is to prove the monotonicity of topological overlap for simplicial complexes and metric spaces with appropriate higher connectivity conditions ($\S\ref{sec:mono}$).

\subsection{Definition for complexes}\label{sec:defnsTO}

\begin{definition}\label{defn:TO} Let $Z\in\fS$. Given a continuous map $f\in C(Z,\R^q)$ we define the \textbf{topological overlap constant} of $f$ to be
\[
 \sOv(f) = \max_{z\in\R^q}\abs{\setcon{\sigma\in\CZ}{z\in f(\sigma)}}
\]
For each $q\geq 1$, the \textbf{$q$-dimensional topological overlap} constant of $Z$ is
\[
 \sTO^q(Z)=\min\setcon{\sOv(f)}{f\in C(Z,\R^q)}.
\]
\end{definition}

We then take a maximum over subcomplexes of a given size, in analogy with the definition of the separation profile.

\begin{definition}\label{defn:TOprof} Let $X\in\bS$. The \textbf{$q$-dimensional topological overlap profile} of $X$ is
\[
 \sTO^q_X(r)=\max\setcon{\sTO^q(Z)}{Z\leq X,\ |Z|\leq r}.
\]
\end{definition}

\begin{remark} It is completely natural to extend the above definitions to finite dimensional CW-complexes, counting cells rather than simplices. A seemingly natural class to consider is complexes with the property that some iterated barycentric subdivision is a simplicial complex (a class which includes all cube complexes, regular CW complexes and all almost-regular CW complexes in the sense of Dru\c{t}u-Kapovich \cite{DrutuKap}). The resulting simplicial complex has bounded degree whenever the cardinality of the set of all closed cells intersecting a given cell can be bounded independent of the chosen cell - this excludes, for example, Cayley $2$-complexes of presentations of groups with infinitely many relators. Under this assumption, taking a barycentric subdivision does not change $\sTO^q_X$ up to $\simeq$.
\end{remark}

In this paper we will use the following definition of topological expander, which assumes bounded degree.

\begin{definition}\label{defn:topexp} Given $\Delta\in\N$ and $\varepsilon>0$, a family of finite $q$-dimensional simplicial complexes $(Z_n)_n$ is called a \textbf{$q$-dimensional $(\Delta,\varepsilon)$-topological expander} if
\begin{itemize}
 \item $|Z_n|\to \infty$ as $n\to\infty$,
 \item $\sup_n\deg(Z_n)\leq \Delta$, for every $n$, and
 \item $\sTO^q(Z_n) \geq \varepsilon|Z_n|$ for every $n$.
\end{itemize}
\end{definition}

We can now give the easy proof of Theorem \ref{thm:TOprops}(i).

\begin{proof}[Proof of Theorem \ref{thm:TOprops}(i)]
    Suppose $X$ is a bounded degree simplicial complex of dimension at most $q$. If there is a $(\Delta,\varepsilon)$-topological expander $(Z_n)$ such that, for each $n$, $Z_n\leq X$, then, for every $n$
    \[
        \sTO^q_X(|Z_n|) \geq \sTO^q(Z_n) \geq \varepsilon|Z_n|
    \]
As $|Z_n|\to\infty$, we have $\liminf_{r\to \infty}\frac1r\sTO^q_X(r) \geq \varepsilon > 0$.

Now suppose $\liminf_{r\to \infty}\frac1r\sTO^q_X(r) \geq \varepsilon > 0$, so there exists an infinite subset $I\subseteq \N$ such that $\sTO^q_X(r) \geq \varepsilon r$ for all $r\in I$. Unpacking Definition \ref{defn:TO}, this means that for every $r\in I$, there is some $Z_r\leq I$ such that 
\[
    \sTO^q(Z_r) \geq \varepsilon r \geq \varepsilon|Z_r|.
\]
It remains to prove that $\{|Z_r|\}$ is not bounded. Suppose for a contradiction that $|Z_r|\leq N$ for all $r\in I$. Any simplicial complex with $N$ $0$-simplices and degree $\Delta$ contains at most $2^\Delta N$ simplices, thus for all $r\in I$, we have
\[
    2^\Delta N  \geq \sTO^q(Z_r) \geq \varepsilon r.
\]
which contradicts the fact that $I\subset \N$ is infinite.
\end{proof}

\subsection{Covering definition}\label{sec:covdim}
In this section we give a second definition of the topological overlap profile which can be defined for metric spaces with bounded geometry, the two definitions will be equivalent for simplicial complexes equipped with the metric \eqref{eq:scmetric} (cf.\ \S\ref{sec:equivdefn}).

\begin{definition} Let $(X,d)$ be a metric space and let $Z\subseteq X$. The $1$\textbf{-covering volume} of $Z$ ($\Cov^1(Z)$) is the minimal $n$ such that $Z$ is contained in a union of $n$ closed balls of radius $1$. If no such $n$ exists we write $\Cov^1(Z)=+\infty$. We say $X$ has \textbf{bounded geometry} if, for every $r>0$,
\[
 \sup_{x\in X} \Cov^1(B(x;r))<+\infty
\]
where $B(x;r)=\setcon{y\in X}{d(x,y)\leq r}$ is the closed ball of radius $r$ centred at $x$.
\end{definition}

The idea behind the second definition is to measure size in terms of covering volume instead of number of simplices.

\begin{definition}\label{defn:cTO} Let $Z$ be a metric space with finite $1$-covering volume. Given a function $f\in C(Z,\R^q)$ and $z\in\R^q$ we define the \textbf{metric overlap of $f$ at $z$} to be $\Ov_f(z)=\Cov^1(f^{-1}(z))$.

The \textbf{metric overlap of $f$} is $\Ov(f)=\max_{z\in\R^q}\Ov_f(z)$.

The \textbf{$q$-dimensional metric topological overlap} of $Z$ is 
\[
	\TO^q(Z)=\min\setcon{\Ov(f)}{f\in C(Z,\R^q)}.
\]
\end{definition}

\begin{definition}\label{defn:cTOprof} Let $X$ be a metric space. The \textbf{$q$-dimensional metric topological overlap} profile of $X$ is 
\[
\TO^q_X(r)=\max\setcon{\TO^q(Z)}{Z\subseteq X,\ \Cov^1(Z)\leq r}.
\]
\end{definition}

\subsection{Equivalence of definitions}\label{sec:equivdefn}

The main result of this section is to prove the equivalence of the two definitions of topological overlap profile for simplicial complexes.

\begin{theorem}\label{thm:defnequiv} Let $X\in\bS$ be equipped with the metric \eqref{eq:scmetric}. For every $q$,
\[
 \sTO^q_X(r)\simeq \TO^q_X(r).
\]
\end{theorem}

% We begin with a useful lemma.

% \begin{lemma}\label{lem:deltaX} Let $X\in\bS$. Every closed simplex $\sigma$ has non-trivial intersection with at most $(\dim(X)+1)2^{\dim(X)}$ other closed simplices.
% \end{lemma}
% \begin{proof}
%  We will work in the underlying simplicial complex. Each $s\in S_X$ appears in at most $\binom{\deg(X)}{k}$ elements of $\mathcal S_X$ with cardinality $k+1$. Thus, $s$ appears in at most
%  \[
%   \sum_{k=0}^{\dim(X)} \binom{\deg(X)}{k} \leq 2^{\deg(X)}
%  \]
%  different simplices in $X$. Since every simplex contains at most $(\dim(X)+1)$ $0$-simplices, each simplex in $\mathcal S_X$ intersects at most $(\dim(X)+1) 2^{\deg(X)}$ other simplices.
% \end{proof}

% For simplicity in what follows, given $X\in\bS$ we write $\Delta_X=(\dim(X)+1) 2^{\deg(X)}$.

We will prove Theorem \ref{thm:defnequiv} as an immediate consequence of the following two propositions, which we prove in a more general setting.

\begin{proposition}\label{prop:cTOtoTO} Let $X$ be a metric space and let $Z\in\bS$. If there is a continuous map $\phi:X\to Z$ and a constant $M$ such that 
\begin{itemize}
\item[\textrm{(a)}] for every ball $B$ of radius $1$ in $X$, $\phi(B)$ intersects at most $M$ closed simplices in $Z$, and
\item[\textrm{(b)}] for every $\sigma \in \CZ$ we have $\Cov^1(\phi^{-1}(\sigma))\leq M$,
\end{itemize} 
then for all $q$,
\[
 \TO^q_X(r) \lesssim \sTO^q_Z(r).
\]
\end{proposition}
\begin{proof}
Let $X'\subseteq X$ satisfy $\Cov^1(X')\leq r$ and let $Z'$ be the subcomplex of $Z$ consisting of all closed simplices which intersect $\phi(X')$. By assumption (a) $|Z'|\leq Mr$. Choose $f\in C(Z',\R^q)$ with $\sOv(f)\leq\TO^q(Mr)$. Using assumption (b), we have
\[
    \Cov^1(\phi^{-1}f^{-1}(z))\leq M\abs{\setcon{\sigma\in \CX'}{z\in f(\sigma)}} \leq M\sOv(f).
\]
As this holds for all $z\in\R^q$, we deduce that
\[
\Ov(\phi\circ f) \leq M\sOv(f) \leq M\sTO^q(Mr).
\]
As $X'$ was chosen arbitrarily, $\TO^q_X(r) \leq M\sTO^q_Z(Mr)$.
\end{proof}

\begin{proposition}\label{prop:TOtocTO}  Let $X$ be a metric space and let $Z\in\bS$. If there is a continuous map $\psi:Z\to X$ and a constant $N$ such that 
\begin{itemize}
\item[\textrm{(c)}] for every $\sigma \in \CZ$, $\Cov^1(\psi(\sigma))\leq N$, and
\item[\textrm{(d)}] for every ball $B$ of radius $1$ in $X$, $\psi^{-1}(B)$  intersects at most $N$ closed simplices in $Z$,
\end{itemize} then for all $q$,
\[
 \sTO^q_Z(r) \lesssim \TO^q_X(r).
\]
\end{proposition}
\begin{proof}
Let $Z'\leq Z$ with $|Z'|\leq r$ and let $X'=\psi(Z')$. By assumption (c) and Fact \ref{fact:simplicebound} we have $\Cov^1(X')\leq N|\CZ'|\leq \Delta_ZN|Z'|\leq \Delta_Z Nr$. Set $N'=\Delta_Z N$.

Choose $g\in C(X',\R^q)$ with $\Ov(g)\leq \TO^q_{X}(N'r)$. By assumption (d), we have
\[
\sOv(\psi\circ g) \leq N\Ov(g) \leq N\TO^q_X(N'r).
\]
As $Z$ was chosen arbitrarily, $\sTO^q_X(r) \leq N\TO^q_Z(N'r)$.
\end{proof}

\begin{proof}[Proof of Theorem \ref{thm:defnequiv}]
    Let $Z\in \bS$ and let $X$ be $Z$ equipped with the metric $\eqref{eq:scmetric}$. Each simplex is contained in a ball of radius $1$ and each ball of radius $1$ intersects a uniformly bounded number of simplices depending only on the degree and dimension of the complex. Therefore the hypotheses of Propositions \ref{prop:cTOtoTO} and \ref{prop:TOtocTO} are satisfied and the result follows immediately.
\end{proof}

\subsection{Monotonicity under regular maps}\label{sec:mono}
The goal of this section is to prove the key monotonicity result, Theorem \ref{thm:TOmono}. We start with monotonicity results for $\sTO$ and $\TO$ under continuous functions.

\begin{proposition}\label{prop:TOmonotone} Let $X,Y\in\bS$ and let $\phi:X\to Y$ be continuous. If there is some $M$ such that for every $\sigma\in\CX$ and $\sigma'\in\CY$, $\phi(\sigma)$ and $\phi^{-1}(\sigma')$ intersect at most $M$ closed simplices in $Y$ and $X$ respectively, then, for every $q$
\[
 \sTO^q_X(r) \lesssim \sTO^q_Y(r).
\]
\end{proposition}
\begin{proof}
Let $Z\leq X$ satisfy $|Z|\leq r$, so $|\CZ|\leq \Delta_X|Z|$ and let $Z'$ be the subcomplex of $Y$ given by the union of all closed simplices which intersect $\phi(Z)$. By assumption and by Fact \ref{fact:simplicebound}, $|Z'|\leq M\Delta_X\Delta_Y|Z|$. Set $M'=M\Delta_X\Delta_Y$. Choose $f\in C(Z',\R^q)$ with $\sOv(f)\leq \sTO^q_Y(M'r)$. Now $(f\circ \phi)|_Z\in C(Z,\R^q)$ and
\[
 \sOv((f\circ \phi)|_Z) \leq M\sOv(g).
\]
As $Z$ was chosen arbitrarily, $\sTO^q_X(r) \leq M\sTO^q_Y(M'r)$.
\end{proof}

\begin{proposition}\label{prop:cTOmonotone} Let $X,Y$ be metric spaces and let $\psi:X\to Y$ be continuous. If there is some $N$ such that
\[
 \Cov^1(\psi(B_X))\leq N \quad \textup{and} \quad \Cov^1(\psi^{-1}(B_Y))\leq N
\]
whenever $B_X$ (respectively $B_Y$) is a closed ball of radius $1$ in $X$ (respectively $Y$), then, for every $q$
\[
 \TO^q_X(r) \leq N\TO^q_Y(Nr).
\]
\end{proposition}
\begin{proof}
Let $Z\subset X$ satisfy $\Cov^1(Z)\leq r$. By definition $\Cov^1(\psi(Z))\leq Nr$. Choose $g\in C(\psi(Z),\R^q)$ with $\Ov(g)\leq\TO^q(Nr)$. Now, $(g\circ \psi)|_Z\in C(Z,\R^q)$ , and
\[
\Ov(g\circ \psi) \leq N\Ov(g) \leq N\TO^q(f(Z)).
\]
As $Z$ was chosen arbitrarily, $\TO^q_X(r) \leq N\TO^q_Y(Nr)$.
\end{proof}

We are now ready to prove Theorem \ref{thm:TOmono}, which follows from Proposition \ref{prop:TOmonotone} and a typical ``connect-the-dots'' argument

\begin{proof}[Proof of Theorem \ref{thm:TOmono}] 
Let $X,Y\in\bS$ and suppose $Y$ is uniformly $k$-connected for every $1\leq k \leq \dim(X)$. Let $\psi:X^0\to Y^0$ be a regular map. Following \cite[Proposition 9.48]{DrutuKap}, there is a constant $C$ and a continuous extension $\phi:X\to Y$ of $\psi$ such that the image of each $k$-simplex is contained in a ball of radius at most $C$ in $Y$. Since $Y\in\bS$, it follows that the image of any simplex under $\phi$ intersects a uniformly controlled number of simplices in $Y$.

Conversely, given $\sigma'\in \CY$ choose a $0$-simplex $y_0$ contained in $\sigma'$. Now suppose $\sigma\in\CX$ satisfies $\phi(\sigma)\cap\sigma'\neq\emptyset$ and let $x_0$ be any $0$-simplex in $\sigma$. It follows that
\begin{equation}\label{eq:simpdist}
    d_Y(y_0,\phi(x_0))=d_Y(y_0,\psi(x_0))\leq \diam(\sigma')+C 
= \sqrt{2}+C.
\end{equation}
As $Y\in\bS$ and $\psi$ is regular, the number of $x_0\in X^0$ which satisfy inequality $\eqref{eq:simpdist}$ can be bounded from above independent of $y_0$. As $X\in bS$ this gives a uniform upper bound on the number of closed simplices in $X$ which intersect $\phi^{-1}(\sigma')$.

Thus, $\phi$ satisfies the hypotheses of Proposition \ref{prop:TOmonotone}.
\end{proof}

\subsection{Monotonicity in $q$}

We record the following obvious monotonicity.

\begin{proposition} Let $X\in \bS$. For every $1\leq q \leq q'$, and every $r$
\[
 \sTO^{q'}_X(r)\leq \sTO^{q}_X(r) \quad \textup{and} \quad \TO^{q'}_X(r)\leq \TO^{q}_X(r).
\]
\end{proposition}

The following na\"ive bound is actually a key part of the proof of Theorem \ref{thm:TOprops}(iii).

\begin{proposition}\label{prop:monotoneincdim} Let $X\in\bS$. For all $q,k\geq 1$ we have 
\[
\sTO^{q+1}_{X^{\leq k+1}}(r)\lesssim \sTO^q_{X^{\leq k}}(r).
\]
\end{proposition}
\begin{proof}
Let $Z\leq X$ with $|Z|\leq r$, and let $f\in C(Z,\R^q)$ satisfy $\sTO^q(Z^{\leq k})=\sOv(f)$.

Each $\sigma\in \CZ^{k+1}$ is homeomorphic to a closed unit $(k+1)$-disc via a homeomorphism which maps the $k$-skeleton of $\sigma$ to the boundary of the disc. For each $\sigma$, choose a continuous extension $g_\sigma:D^{k+1}\to\R^q$ of $f|_{\sigma^{\leq k}}:S^{k+1}\to\R^q$, and fix an injection $\iota: \CZ^{k+1}\to (0,1)$.

We identify $D^{k+1}$ with the Euclidean ball of radius 1 in polar coordinates:
\[
 D^{k+1}=\setcon{(R,\theta,\psi_1,\ldots,\psi_{k-1})}{R\in [0,1],\ \theta\in[0,2\pi),\ \psi_i\in[0,\pi)}
\]

We now define $g: Z^{\leq k+1}\to \R^{q+1}$ as follows. For $x\in Z^{\leq k}$, we set $g(x)=(f(x)_1,\ldots,f(x)_q,0)$. Then for each $\sigma\in \CZ^{k+1}$ we set $g((R,\theta,\psi_1,\ldots,\psi_{k-1}))$ to be
\[
 \left\{\begin{array}{lll}
 (f|_{\partial \sigma}(1,\theta,\psi_1,\ldots,\psi_{k-1}),2(1-R)\iota(\sigma)) & \textup{if} & R\in[\frac12, 1], \\
 (g_\sigma(2R,\theta,\psi_1,\ldots,\psi_{k-1}),\iota(\sigma))  & \textup{if} & R\in[0, \frac12].
 \end{array}\right.
\]
%\begin{figure}[h!]
%\centering
%\include{tikzextension}
%\caption{An example with $q=2$ and $k=1$.}
%\end{figure}

Note that $g$ is well-defined, and is a continuous extension of $f$. Let $\underline{z}=(z_1,\ldots,z_{q+1})\in\R^{q+1}$. Suppose the image of $\sigma\in\CZ^{\leq k+1}$ contains a point in $g^{-1}(\underline{z})$. 
By construction, $g^{-1}(\underline{z})$ is contained in the union of all simplices which contain a point in $f^{-1}(z_1,\ldots,z_q)\cup\iota^{-1}(z_{q+1})$. By assumption there are at most $\sOv(f)$ simplices of dimension at most $k$ which contain a point in $f^{-1}(z_1,\ldots,z_q)$. Since there are at most $\deg(Z)$ $(k+1)$-simplices containing any given $k$-simplex and $\iota$ is injective, we have
\[
 \sOv(g) \leq \sOv(f) + \deg(Z)\sOv(f)+1 = (1+\deg(Z))\sOv(f)+1.
\]
As this holds for all $Z$ and $X\in\bS$, we have
\[
 \sTO^{q+1}_{X^{\leq k+1}}(r)\leq (1+\deg(X))\sTO^q_{X^{\leq k}}(r)+1.
\]
Thus, $\sTO^{q+1}_{X^{\leq k+1}}(r)\lesssim\sTO^q_{X^{\leq k}}(r)$.
\end{proof}

We record the following simple corollary of the above result and Theorem \ref{thm:TOprops}(i).

\begin{corollary} Let $X\in \bS$. If $X^{\leq q}$ contains a $q$-dimensional topological expander, then for all $1\leq q'< q$, $X^{\leq q'}$ contains a $q'$-dimensional topological expander.
\end{corollary}

\subsection{Passing to lower skeleta}

For small $q$, we may compute $\sTO^q_X$ using only the $q$-skeleton of $X$.

\begin{proposition}\label{prop:12skeleton2} Let $X\in\bS$. For $q=1,2$, we have 
\[
	\sTO^q_{X^{\leq q}}(r)\simeq \sTO^q_{X}(r).
\] 
\end{proposition}
The key tool required for this proposition is the following. For $q=1,2$ and $n> q$, the image of a continuous map $f:S^{n-1}\to\R^q$ is null-homotopic, i.e.\ there is a continuous extension $\overline f:D^n\to\R^q$ of $f$ such that $\textrm{im} \overline{f}=\textrm{im} f$. This is the intermediate value theorem for $q=1$, for $q=2$ was proved by Cannon-Conner-Zastrow \cite{CCZ} and for $q\geq 3$ it is false, since there is a continuous surjective map (the Hopf map) $f:S^3\to S^2 \subset\R^3$, which cannot be extended to a continuous map $\overline f:D^3\to S^2$.

\begin{proof}
It is immediate that $\sTO^q_{X^{\leq q}}(r) \leq \sTO^q_{X}(r)$.

Choose $Z\leq X$ with $|Z|\leq r$ and let $f\in C(Z^{\leq q},\R^q)$. By the above there is a continuous extension $g\in C(Z,\R^q)$ such that the image of every $n$-simplex with $n>q$ is contained in the image of its $q$-skeleton.

Let $z\in\R^q$. Every simplex in $\CZ\setminus\CZ^{\leq q}$ whose image contains $z$ must intersect some simplex in $\CZ^{\leq q}$ whose image contains $z$. Since $X\in\bS$ we have $\sOv_{g}(z)\leq \Delta_X\Ov_f(z)$. As $z$ and $Z$ were chosen arbitrarily, we have
\[
 \sTO^q_{X}(r) \leq \Delta_X \sTO^q_{X^{\leq q}}(r).
\]
($\Delta_X$ was defined in Fact \ref{fact:simplicebound} as an upper bound for the maximal number of closed simplices which intersect a given closed simplex). 
\end{proof}
Although the above proof fails, it remains open whether $\sTO^3_X(r)\simeq \sTO^3_{X^{\leq 3}}(r)$ for every $X\in \bS$.

\subsection{Bounded topological overlap}

We give two sufficient criteria for bounded topological overlap. The first of these is Theorem \ref{thm:TOprops}(iii). We recall the statement for convenience.

\begin{theorem*} Let $X$ be any finite dimensional simplicial complex. For all $q>\dim(X)$ we have $\sTO^q_X(r)\simeq 1$.
\end{theorem*}
\begin{proof}
Firstly, $\sTO^2_{X^{\leq 1}}(r)\leq 2$ as every finite graph can be continuously mapped to $\R^2$ so that the preimage of every point where the map is not injective is a union of two points contained in the interior of two edges. We then apply Proposition $\ref{prop:TOmonotone}$ to see that
\[
 \sTO^{q-(\dim(X)-1)}_{X^{\leq 1}}\simeq 1
\]
Applying Proposition $\ref{prop:monotoneincdim}$ $(\dim(X)-1)$ times, we see that $\sTO^q_X(r)\simeq 1$.
\end{proof}

The second is a simple consequence of the above result and the definition of asymptotic dimension.

\begin{corollary}\label{cor:bddtovasdim} Let $X$ be a geodesic metric space with bounded geometry. If $\asdim(X)\leq n$ then $\TO^q_X(r)\simeq 1$ for all $q>n$.
\end{corollary}
\begin{proof}
Since $X$ has asymptotic dimension at most $n$ there is a $1$-Lipschitz (and hence continuous) map $\phi:X\to K$, where $K$ is a simplicial complex with $\dim(K)\leq n$, such that preimages of simplices have diameter bounded by some $M$ (cf.\ \cite[\S$1.E_1$, p21--22]{Gromov93})\footnote{the assumption that $X$ is geodesic is necessary to guarantee that $\phi$ is Lipschitz when $K$ is equipped with the metric \eqref{eq:scmetric}}. As $X$ has bounded geometry, the $K$ constructed in this proof necessarily has bounded degree, hence $K\in\bS$.

Now let $Z\subset X$ with $\Cov^1(Z)\leq r$. As $\phi$ is $1$-Lipschitz, there is a finite subcomplex $K_Z$ of $K$ which contains $\phi(Z)$. By Theorem \ref{thm:TOprops}(iii), there is a continuous map $f:K_Z\to\R^q$ with $\sOv(f)\leq M'$. It follows that for every $z\in\R^q$, $\phi^{-1}(f^{-1}(z))$ is contained a union of at most $M'$ sets of diameter at most $M$. Since $X$ has bounded geometry, we get a uniform upper bound on $\Cov^1(\phi^{-1}(f^{-1}(z)))$ which is independent of $z$, $Z$ and $r$. 
\end{proof}

\section{The dimension 1 case}\label{sec:dim1}

The goal of this section is to prove Theorem \ref{thm:compTO1cutwidth}. As this section primarily concerns graphs, we will refer to $0$- and $1$- simplices as vertices and edges respectively, we also write $V\Gamma=\Gamma^0$ and $E\Gamma=\Gamma^1$.

\subsection{$\sTO^1$ and cutwidth}

Let us recall the definition of the cutwidth profile from \cite{HHKL23}.

\begin{definition}\label{defn:cutwidth} Let $\Gamma$ be a finite graph with $r$ vertices. The \textbf{cutwidth} of $\Gamma$ (denoted $\cw(\Gamma)$) is the minimum value of 
\[
\max_{i}|\{vw \in E\Gamma:\ \sigma(v) < i \leq \sigma(w)\}|
\] 
over all bijections $\sigma:V\Gamma\to\{1,\ldots,r\}$.

We define the \textbf{cutwidth profile} of a graph $X$ by
\[
 \cw_X(r)=\max\setcon{\cw(\Gamma)}{\Gamma\leq X,\ |\Gamma|\leq r}.
\]
\end{definition}

Our next target is Theorem \ref{thm:compTO1cutwidth}.

\begin{lemma}\label{lem:TO1cw} For every finite graph $\Gamma$, 
\[
	\cw(\Gamma) \leq \sTO^1(\Gamma) \leq \cw(\Gamma) + \deg(\Gamma)+1.
\]
\end{lemma}
\begin{proof} Every bijection $\sigma:V\Gamma\to \{1,\ldots,r\}$ naturally extends to a continuous map $\overline\sigma:\Gamma\to [1,r]$ so that the image of the edge $vw$ is the interval $[\sigma(v),\sigma(w)]$. For each $z\in \R$, $\overline\sigma^{-1}(z)$ contains a single point in each edge $vw$ with $\sigma(v) \leq z \leq \sigma(w)$ and (if $z\in\{1,\ldots,r\}$) the vertex $v'=\sigma^{-1}(z)$ and all edges with end vertex $v'$.

Thus $\sTO^1(\Gamma) \leq \sOv(\overline\sigma)\leq \cw(\Gamma) + 1 + \deg(\Gamma)$.

Conversely, let $g\in C(\Gamma,\R)$. Enumerate the vertices of $\Gamma$ by $v_1,\ldots,v_r$ so that $g(v_i)\leq g(v_j)$ whenever $i\leq j$. Define $\sigma: V\Gamma\to\{1,\ldots,r\}$ by $\sigma(v_j)=j$. By the intermediate value theorem $g^{-1}(i)$ contains a point in every edge $v_jv_k$ with $j<i\leq k$. Thus
\[
\sOv(g) \geq \max_i|\{vw \in E\Gamma:\ \sigma(v) < i \leq \sigma(w)\}|\geq \cw(\Gamma).
\]
As $g$ was chosen arbitrarily, we have $\sTO^1(\Gamma)\geq \cw(\Gamma)$.
\end{proof}

We can now prove Theorem \ref{thm:compTO1cutwidth}, which we restate for convenience.

\begin{theorem*} Let $X$ be a bounded degree simplicial complex, and let $X^{\leq 1}$ denote the $1$-skeleton of $X$. We have
\[    
    \sTO^1_X(r)\simeq \cw_{X^{\leq 1}}(r).
\]
\end{theorem*}
\begin{proof}
    Let $X\in \bS$. By Proposition \ref{prop:12skeleton},
    \[
        \sTO^1_X(r)\simeq \sTO^1_{X^{\leq 1}}(r).
    \]
    As $X^{\leq 1}$ is a bounded degree graph, Lemma \ref{lem:TO1cw} implies that
    \[
        \cw_{X^{\leq 1}}(r) \leq \sTO^1_{X^{\leq 1}}(r)\leq \cw_{X^{\leq 1}}(r)+1+\deg(X).
    \]
    Thus, $\sTO^1_X(r)\simeq \cw_{X^{\leq 1}}(r)$, as required.
\end{proof}

\subsection{Topological expanders contain graph expanders}

In this section we prove that every bounded degree $1$-dimensional topological expander (cf.\ Definition \ref{defn:topexp}) necessarily contains a bounded degree graphical expander. We recall the definition of graphical expanders for convenience.

\begin{definition} Let $\Gamma$ be a graph with $\abs{\Gamma}=n$. The vertex-boundary of a subset $A\subseteq V(\Gamma)$ - denoted $\partial A$ - is the set of all vertices in $V(\Gamma)\setminus A$ which are neighbours of some vertex of $A$. The (vertex) \textit{Cheeger constant} of $\Gamma$ is given by
\[
h(\Gamma)=\min\varsetcon{\frac{\abs{\partial A}}{\abs{A}}}{0<\abs{A}\leq n/2}.
\]

Given $\Delta'\in\N$ and $\varepsilon'>0$, a family of finite graphs $(\Gamma_n)_n$ is called a \textbf{$(\Delta',\varepsilon')$-graphical expander} if
\begin{itemize}
 \item $|V\Gamma_n|\to \infty$ as $n\to\infty$,
 \item $\sup_n\deg(\Gamma_n)\leq \Delta'$, and
 \item $\inf_n h(\Gamma_n)\geq \varepsilon'$.
\end{itemize}
\end{definition}

We will prove that every topological expander contains a graphical expander using the separation profile. For convenience we also recall this definition from \cite{BenSchTim-12-separation-graphs}.

\begin{definition}
    Let $\Gamma$ be a finite graph. The cutsize of $\Gamma$, $\cut(\Gamma)$, is the minimal cardinality of a subset $S\subseteq V\Gamma$ such that every connected component of $\Gamma-S$ contains at most $\frac12|V\Gamma|$ vertices.

    Let $X$ be a graph. We define the separation profile of $X$ to be the function
    \[
        \sep_X(r) = \max\setcon{\cut(\Gamma)}{\Gamma\leq X,\ |V\Gamma|\leq r}
    \]
\end{definition}

Note that for every graph $X$, $\sep_X(r)\leq \lceil \frac{r}{2}\rceil$. The separation and cutwidth profiles satisfy the following relation \cite[Lemma 4.1]{HHKL23}.

\begin{lemma}\label{lem:cwsep} Let $X$ be a graph with maximal degree $\Delta$. For every $r$,
\[
    \cw_X(r) \leq \cw_X(r/2) + \Delta\sep_X(r)
\]
\end{lemma}

Our next goal is to prove Corollary \ref{cor:topexp-graphexp}, which we now restate with explicit bounds on constants.

\begin{corollary}\label{cor:topexp-graphexp2}
Let $Z=\bigsqcup_n Z_n\in \bS$, where $(Z_n)_{n\in\N}$ is a $1$-dimensional $(\Delta,\varepsilon)$-topological expander. For all sufficiently large $n$ there is a subgraph $\Gamma_n\leq Z_n$ such that 
\begin{itemize}
\item $|\Gamma_n| \geq \frac{\varepsilon}{8\log_2(12\Delta/\varepsilon)}|Z_n|$, and
\item $(\Gamma_n)_{n\in\N}$ is a $(\Delta,\frac{\varepsilon}{8\Delta\log_2(12\Delta/\varepsilon)})$-graphical expander.
\end{itemize}
\end{corollary}
\begin{proof}[Proof of Corollary \ref{cor:topexp-graphexp2}]
Set $r_n=|Z_n|$ and choose $k$ so that $2^{k-1}<r_n \leq 2^k$. Applying Lemma {lem:TO1cw} and repeatedly using Lemma \ref{lem:cwsep}, we have
\begin{align*}
	\varepsilon r_n & \leq \sTO^1(Z_n) \\ & \leq 1+\Delta+\cw_{Z_n}(r_n) \\ & \leq 1+\Delta+\Delta\sum_{i=0}^{k} \sep_{Z_n}(2^{i-k}r_n)
	\\ & \leq 1+\Delta+\Delta\sum_{i=0}^{\left\lfloor \log_2\left(\frac{\varepsilon r_n}{3\Delta}\right)\right\rfloor} \sep_{Z_n}(2^{i-k}r_n)+\Delta\sum_{i=\left\lceil \log_2\left(\frac{\varepsilon r_n}{3\Delta}\right)\right\rceil}^{\lceil \log_2(r_n)\rceil} \sep_{Z_n}(2^{i-k}r_n)
        \\ & \leq 1+\Delta+\Delta\sum_{i=0}^{\left\lfloor \log_2\left(\frac{\varepsilon r_n}{3\Delta}\right)\right\rfloor} 2^{i-1}+\Delta\left(\lceil \log_2(r_n)\rceil-\left\lceil \log_2\left(\frac{\varepsilon r_n}{3\Delta}\right)\right\rceil+1\right) \sep_{Z_n}(r_n)
    \\ & \leq 1 +\Delta + \Delta 2^{\left\lfloor \log_2\left(\frac{\varepsilon r_n}{3\Delta}\right)\right\rfloor} + \Delta\left(\log_2\left(\frac{3\Delta}{\varepsilon}\right)+2\right)\sep_{Z_n}(r_n)
    \\ & \leq 1+\Delta+\frac23\varepsilon r_n+\Delta(2+\log_2(3\Delta/\varepsilon)) \sep_{Z_n}(r_n).
\end{align*}
Here, the fifth inequality uses the fact that $\sep_X(r)\leq \lceil \frac{r}{2}\rceil$ for every graph $X$ and every $r$, so $\sep_{Z_n}(2^{i-k}r_n)\leq \sep_{Z_n}(2^i)\leq 2^{i-1}$. Thus, for all $n$ such that $r_n \geq 12(\Delta +1)\varepsilon^{-1}$,
\[
    \frac{\varepsilon}{4}r_n \leq \frac{\varepsilon}{3}r_n - 1 - \Delta \leq \Delta(2+\log_2(3\Delta/\varepsilon)) \sep_{Z_n}(r_n)
\]
Rearranging, we see that
\[
 \sep_{Z_n}(r_n) \geq \frac{\varepsilon}{4\Delta\log_2(12\Delta/\varepsilon)}r_n.
\]
Therefore, there is a subcomplex $\Gamma'_n\leq Z_n$ with
\begin{align*}
 |\Gamma'_n| \geq \cut(\Gamma'_n) \geq \frac{\varepsilon}{4\Delta\log_2(12\Delta/\varepsilon)}r_n
\end{align*}
so by \cite[Proposition 2.4]{HumSepExp}, there is another subcomplex $\Gamma_n\leq Z_n$ with
\begin{align*}
 |\Gamma_n| \geq \frac{\varepsilon}{8\Delta\log_2(12\Delta/\varepsilon)}r_n \\
 h(\Gamma_n) \geq \frac{\varepsilon}{8\Delta\log_2(12\Delta/\varepsilon)}.
\end{align*}
Thus $(\Gamma_n)_n$ is a $\left(\Delta,\frac{\varepsilon}{8\Delta\log_2(12\Delta/\varepsilon)}\right)$-graphical expander.
\end{proof}

\subsection{Bounds on $\TO^1$ from cutwidth profiles}\label{sec:TO1conseq}

Finally in this section, we quickly review known bounds on cutwidth profiles and the consequences for topological overlap.

\begin{corollary}\label{cor:planarTO1} Let $H$ be a finite graph and suppose $X$ is a bounded degree graph with no $H$-minor. Then $\sTO^1_X(r)\lesssim r^\frac12$.
\end{corollary}
\begin{proof}
Combining \cite[Theorem 1.1]{WoodTelle} and \cite[Theorem 3]{Djidjev2003CrossingNA} we see that $\cw_X(r)\lesssim r^{1/2}$. Thus, by Theorem \ref{thm:compTO1cutwidth},
\[
 \sTO^1_X(r) \lesssim r^{1/2}. \qedhere
\]
\end{proof}

\begin{corollary}\label{cor:fgnilpTO1} Let $X=H\times N$ where $H$ is either a point or a rank 1 symmetric space whose boundary has conformal dimension $Q$ and $N$ is a compactly generated locally compact group with polynomial growth of degree $d$. Then
\begin{align*}
 \TO^1_X(r) \simeq r^{1-1/d} & \ \textup{if}\ |H|=1,
 \\
 \ln(1+r) \lesssim \TO^1_X(r) \lesssim \ln(1+r)^2  & \ \textup{if}\ H=\HH^2,\ d=0,
 \\
 \TO^1_X(r) \simeq r^{1-1/(d+1)}\ln(1+r)^{1/(d+1)} & \ \textup{if}\ H=\HH^2,\ d\geq 1,
 \\
 \TO^1_X(r)\simeq r^{1-1/(Q+d)} &\  \textup{otherwise.}
\end{align*}
\end{corollary}
\begin{proof}
This follows immediately from Theorem \ref{thm:compTO1cutwidth}, \cite[Corollary 4.3]{HHKL23}, \cite[Theorem 7]{HumeMackTess-Pprof} and \cite[Theorem 1.12]{HumeMackTess-PprofLie}. 
\end{proof}
Note that the result $\sep_X(r)\simeq r^{1-1/d}$ holds for a more general class of polynomially-growing graphs \cite[Theorem 1.2]{LeCozGour}.

\begin{corollary} Let $X$ be a horocyclic product of two $3$-regular trees, a solvable Lie group of exponential growth or a Riemannian symmetric space whose noncompact factor has rank at least $2$. Then
\[
\TO^1_X(r)\simeq r/\log(r).
\]
\end{corollary}
\begin{proof} 
This follows immediately from Theorem \ref{thm:compTO1cutwidth}, \cite[Corollary 4.3]{HHKL23} and \cite[Theorem 1.11]{HumeMackTess-PprofLie}.
\end{proof}

\section{Higher dimensions}\label{sec:highdim}
We now move to calculating topological overlap of familiar spaces in higher dimensions.

\subsection{Euclidean spaces}\label{sec:Eucbds}

The goal of this section is to calculate the topological overlap of Euclidean spaces. Let us recall the statement of Theorem \ref{thm:TOEuc} from the introduction:

\begin{theorem*} For each $q\leq n$ we have
\[
 \TO^q_{\R^n}(r)\simeq r^{1-q/n}.
\]
\end{theorem*}

\subsubsection{Lower bound}

Our first goal is the lower bound, which we prove using the covering definition of topological overlap and one of several proofs of the waist inequality due to Gromov (cf.\ \cite[Theorem 5]{Guth-waist}).

\begin{theorem}\label{thm:Gromov_waist} For each $k>q$ there exists a constant $\beta'_{k,q}>0$ such that the following holds. Suppose $F:S^k\to \R^q$ is a continuous map. For every $r>0$ there is some $y\in\R^q$ such that $F^{-1}(y)$ cannot be covered by fewer than $\beta'_{k,q}r^{-(k-q)}$ balls of radius $r$ in $S^k$.
\end{theorem}

Our goal is to derive a comparable result for continuous maps from the ball of radius $R$ in $\R^k$ to $\R^q$.

For each $R\geq 1$, we define $RS^k=\setcon{(x_0,\ldots,x_k)\in\R^{k+1}}{\sum x_i^2=R^2}$ and $RD^k=\setcon{(x_0,\ldots,x_{k-1},0)\in\R^{k+1}}{\sum x_i^2\leq R^2}$. We write $S^k$ and $D^k$ for $1S^k$ and $1D^k$ respectively.

Define $\pi:S^k\to D^k$ so that $\pi(x_0,\ldots,x_k)$ is the unique point on $D^k$ lying on the straight line connecting $(x_0,\ldots,x_k)$ to $(0,\ldots,0,-1)$ if $x_k\geq 0$ and $(0,\ldots,0,1)$ if $x_k\leq 0$. This is well-defined as for $x_k=0$, $\pi(x_0,\ldots,x_k)=(x_0,\ldots,x_k)$ using either rule.

Since $\pi$ is biLipschitz when restricted to either half-sphere and $S^k$ is a doubling metric space, there is a constant $C$ such that for every $r>0$, the primage under $\pi$ of a ball of radius $r$ in $D^k$ is contained in a union of at most $C$ balls of radius $r$ in $S^k$.

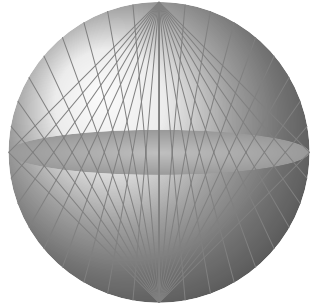
\begin{center}
\begin{tikzpicture}
    \shade[ball color=black!10!white, opacity=0.6] (0,0) circle (2cm);
    \shade[inner color=black!25!white, outer color=black!40!white, opacity=0.6] (0,0) ellipse (2cm and 0.3cm);
	\foreach \i in {0,...,18}{
		\draw[black!50, very thin]  	(10*\i:2cm) -- (0,-2)
							(-10*\i:2cm) -- (0,2);
	}
\end{tikzpicture}
\end{center}

Rescaling this by a parameter $R\geq 1$ we get maps $\pi_R: RD_k\to RS_k$ such that the preimage of every ball of radius $1$ in $RD^k$ is contained in at most $C$ balls of radius $1$ in $RS^k$.

We now record the following simple corollary to the waist inequality above. 

\begin{corollary}\label{cor:Gromov_waist} For each $k>q$ there exists a constant $\beta''_{k,q}>0$ such that the following holds for all $R\geq 1$. Suppose $G:RD^k\to \R^q$ is a continuous map. There is some $y\in\R^q$ such that $G^{-1}(y)$ cannot be covered by fewer than $\beta''_{k,q}R^{(k-q)}$ balls of radius $1$ in $RD^k$.
\end{corollary}
\begin{proof}
Rescaling Theorem \ref{thm:Gromov_waist} by $R=1/r$, we see that for every continuous map $F:RS^k\to \R^q$, there is some $y\in\R^q$ such that $F^{-1}(y)$ cannot be covered by fewer than $\beta'_{k,q}R^{(k-q)}$ balls of radius $1$ in $RS^k$.

Now let $G:RD^k\to\R^q$ be a continuous map and define $F:RS^k\to\R^q$ by $F=\pi_R\circ G$. Suppose for a contradiction, that for every $y\in\R^d$, $G^{-1}(y)$ can be covered by fewer than $\frac1C\beta'_{k,q}R^{(k-q)}$ balls of radius $1$ in $RD^k$. Taking preimages of $\pi_R$, we cover each $F^{-1}(y)$ by fewer than $\beta'_{k,q}R^{(k-q)}$ balls of radius $1$ in $RS^k$, contradicting Corollary \ref{cor:Gromov_waist}. The result then follows with $\beta''_{k,q}=C^{-1}\beta'_{k,q}$.
\end{proof}

\begin{proof}[Proof of Theorem \ref{thm:TOEuc} - lower bound]
Calculating $\Cov^1(RD^k)$ exactly (using the Euclidean metric) is a very difficult problem. A sufficient upper bound for our purposes is $(2k)^kR^k$, since a ball of radius $R$ is contained in a cube of sidelength $2R$ and each cube of sidelength $1$ is contained in at most $k^k$ balls of radius $1$. By Corollary \ref{cor:Gromov_waist}, $\TO^q(RD^k)\geq \beta''_{k,q}R^{(k-q)}$, thus
\[
    \TO^q_{\R^k}(r) \geq \frac{\beta''_{k,q}}{(2k)^{k-q}}r^{1-\frac{q}{k}}. \qedhere
\]
\end{proof}

\subsubsection{Upper bound}
The strategy behind the upper bound is as follows. We fix $r\in\N$ and define two cubulations of $\R^k$. The first, denoted $\mathcal C$, has $0$-skeleton $(\frac12+\Z)^k$ and top dimensional cubes $(\frac12+\Z)^k+[0,1]^k$. The second, denoted $\mathcal C_r$, has $0$-skeleton $(r\Z)^k$ and top dimensional cubes $(r\Z)^k+[0,r]^k$. We refer to top-dimensional cubes in the two cubulations as $\mathcal C$-cubes and $\mathcal C_r$-cubes respectively. For each $0\leq l\leq k$ we define $X^{k-l}$ to be the $(k-l)$-skeleton of $\mathcal C_r$. We also define $Y^{k-l}$ to be the union of all the $\mathcal C$-cubes which intersect $X^{k-l}$. The strategy is as follows:
\begin{itemize}
    \item Choose a set $Z$ of at most $r^k$ $\mathcal C$-cubes. 
    \item Find $v\in \Z^k$, such that $(v+Z)\cap Y^{k-q}$ is contained in a union of $\leq \binom{k}{q} r^{k-q}$ cubes of side length $1$. (Without loss, we then translate $Z$ so that $v=0$)
    \item Build continuous functions $\R^k\to \R$ such that the preimage of every point is almost entirely contained in $Y^{k-1}$.
    \item Take the direct sum of $q$ of these functions to find a continuous function $\R^k\to \R^q$ such that the preimage of every point is almost entirely contained in $Y^{k-q}$.
\end{itemize}

Fix an orthonormal basis $\mathcal B = \{e_1,\ldots,e_k\}$ of $\R^k$. Any vector representation of an element of $\R^k$ will be with respect to this ordered basis. We denote by $\pi_i$ the projection of $\R^k$ onto $\langle e_i\rangle$ and by $\hat\pi_i$ the projection of $\R^k$ onto $e_i^\perp$, i.e.
\begin{align*}
\hat\pi_i(x_1,\ldots,x_k) & = (x_1,\ldots,x_{i-1},0,x_{i+1},\ldots,x_k) \\
\pi_i(x_1,\ldots,x_k) & = (0,\ldots,0,x_{i},0,\ldots,0)    
\end{align*}

Fix $r\in\N$ with $r\geq 2$. For this section, a $\mathcal{C}$-\textbf{cube} is a set of the form
\[
 A = \prod_{j=1}^k A_j \subset \R^k
\]
where $A_j=[a_j,a_j+1]$ with $a_j\in\Z+\frac12$. The \textbf{root} of the $\mathcal{C}$-cube $A$ is the point $(a_1,\ldots,a_k)$.  A $\mathcal{C}_r$-\textbf{cube} is a set of the form
\[
 B = \prod_{j=1}^k B_j \subset \R^k
\]
where $B_j=[b_j,b_j+r]$ with $b_j\in r\Z$. The root of the $\mathcal{C}_r$-cube is $(b_1,\ldots,b_k)$.
\medskip

It will be convenient to cover sets using $\mathcal{C}$-cubes rather than metric balls in this section, so we introduce the following

\begin{definition} Let $A\subseteq \R^k$. We denote by $\Cov^c(A)$) the minimal $n$ such that $A$ is contained in a union of $n$ $\mathcal{C}$-cubes. If no such $n$ exists we write $\Cov^c(A)=+\infty$.

A \textbf{pixellation} of $A$ is a minimal union of $\mathcal{C}$-cubes which contains $A$.

We define $\cOv$, $\cTO^q(A)$ and $\cTO^q_{\R^k}$ as in Definitions \ref{defn:cTO} and \ref{defn:cTOprof} replacing $\Cov^1$ by $\Cov^c$.
\end{definition}

Note that every ball of radius $1$ in $\R^k$ (equipped with the Euclidean metric) is contained in a union of at most $3^k$ $\mathcal{C}$-cubes, and each $\mathcal{C}$-cube is contained in a union of at most $k^k$ balls of radius $1$ (for instance, $A$ is covered by the union of balls of radius $1$ whose centres are of the form $(a_1,\ldots,a_k)+\frac1k(n_1,\ldots,n_k)$ with $n_j\in\{0,\ldots,k-1\}$). Therefore, for any $A\subseteq \R^k$, $k^{-k}\Cov^1(A)\leq\Cov^c(A) \leq 3^k\Cov^1(A)$. From this we easily deduce

\begin{lemma} For every $q$, $\cTO^q_{\R^k}(r)\simeq \TO^q_{\R^k}(r)$.
\end{lemma}
For the remainder of the section we will work with $\cOv$ and $\cTO$.
\medskip

Fix $Z'\subset \R^k$ with $\Cov^c(Z')\leq r^k$. The first step is to show that some translate of $Z'$ intersects $Y^{k-q}$ in a relatively small set.

\begin{lemma} Let $Z'\subset \R^k$ with $\Cov^c(Z')\leq r^k$, and let $Z$ be a pixellation of $Z'$. There is some $v=(v_1,\ldots,v_k)\in\Z^k$ such that
\[
 \Cov^c \left((-v+Z) \cap Y^{k-q}\right) \leq \binom{k}{q}r^{k-q}.
\]
\end{lemma}
\begin{proof}
%Fix a pixellation $Z$ of $Z'$ by cubes $C^{m_i}$ with root $(a^{m_1}_1,\ldots,a^{m_s}_k)$.
We have $\vol^k(\frac1rZ)\leq 1$, so, using an averaging argument (see for instance \cite[Proof of Theorem 1]{Guth-widthvol}) there is some $v'=(v_1,\ldots,v_k)\in [0,1]^k$ such that $\vol^{k-q}(\frac1rZ\cap (v'+\frac1rX^{k-q}))\leq \binom{k}{q}$. Rescaling, we see that
 \begin{equation}\label{eq:vol}
 \vol^{k-q}(Z\cap (rv'+X^{k-q}))\leq \binom{k}{q}r^{k-q}.
 \end{equation}
Let $v=([rv'_1],\ldots,[rv'_k])\in\Z^k$. Since $Z$ is a union of cubes,
 \[
  \vol^{k-q}(Z\cap (v+X^{k-q}))\leq \vol^{k-q}(Z\cap (rv'+X^{k-q})),
 \]
with equality except possibly when $rv_i\in\Z+\frac12$ for some $i$. Now $Z\cap (v+X^{k-q})$ is a union of sets of the form
\[
 \prod_{i=1}^k F^i
\]
where exactly $q$ of the $F^i\subset\R$ are singleton sets $\{f^i_j+\frac12\}$ with $f^i_j\in\Z+\frac12$, and the remaining $k-q$ are intervals $[f^i_j,f^i_j+1]$ with $f^i_j\in\Z+\frac12$. As each of these sets contributes exactly $1$ to the volume in \eqref{eq:vol} and $Z$ is contained in the union of the $\mathcal C$-cubes with root $(f^i_j)_{j=1}^k$, we have,
\[
	\Cov^c(Z\cap (v+X^{k-q}))=\Cov^c(Z\cap (v+Y^{k-q}))\leq \binom{k}{q}r^{k-q}. 
\]
Finally, $\Cov^c$ is invariant under translation by elements of $\Z^k$, and, by definition, any $\mathcal C$-cube which intersects $(-v+Z) \cap X^{k-q}$ is contained in $(-v+Z) \cap Y^{k-q}$. Thus
\[
 \Cov^c \left((-v+Z) \cap Y^{k-q}\right)=\Cov^c \left((-v+Z) \cap X^{k-q}\right) \leq \binom{k}{q}r^{k-q}. \qedhere
\]
\end{proof}
In what follows, we will assume that $Z$ is chosen so that $v=0$, i.e.
\[
 \Cov^c \left(Z \cap Y^{k-q}\right) \leq \binom{k}{q}r^{k-q}.
\]

\subsubsection{Defining continuous functions}
Now that we have found a suitable translate of $Z$, we move to the next step of constructing a continuous function $f:\R^k\to\R^q$ such that the preimage of each point is mostly contained in $Y^{k-q}$.

Before commencing with the fine detail, we give a more structured overview of the construction. We will build $q$ functions $f_i:\R^k\to \R$ for $i=1,\ldots,q$ with the property that the preimage of each point is contained in a union of sets of the form
\[
    \lambda_je_j + e_j^{\perp}
\]
together with some ``small'' set $C$ we will describe later. We will choose the values of $\lambda_j$ carefully, so that they satisfy the two conditions:
\begin{itemize}
    \item for every $j$, $\lambda_j\in r\Z+(-\frac12,\frac12)$; and
    \item for $i\neq i'$, and any $j$, we never see the same value of $\lambda_j$ appearing in the preimages of points in $f_i$ and $f_{i'}$.
\end{itemize}
The set $C$ will be contained in the intersection of a translate of a hyperplane $H$ whose dihedral angle with $e_i^\perp$ is in $(0,\frac{\pi}{4})$ with a $(\frac12+\Z)^k$ translate of $[0,r+1]^k$ (or equivalently, the pixellation of a $\mathcal C_r$-cube).

Having done this, we define $f=\bigoplus f_i:\R^k\to\R^q$. By construction, the preimage of any point is contained in a union of intersections of the hyperplanes described above. The careful choice of the coefficients $\lambda_i$ ensures that any intersecting family of $q$ hyperplanes are in general position, and therefore, the preimage is a union of translates of subspaces of dimension $(k-q)$ which are contained in $Y^{k-q}$ together with the intersection of a single translation of a subspace of dimension $(k-q)$ with a cube of sidelength $r+1$. An image of such an $f_i$ (for $k=2$ and $r=4$) is given below:

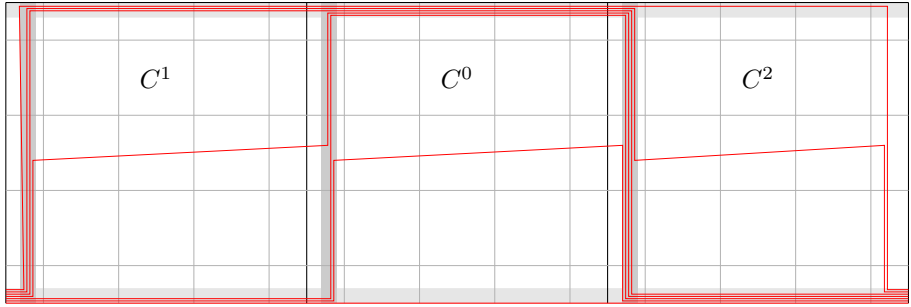
\begin{center}
\begin{tikzpicture}[yscale=1,xscale=1, vertex/.style={draw,fill,circle,inner sep=0.3mm}]
	\clip (0,0) rectangle (12,4);
	\fill[black!10] 	(0,0) rectangle (12,0.2)
	(0,3.8) rectangle (12,4);
	
	\foreach \r in {0,...,3}{
		\fill[black!20] 	%(\r*4-0.4,0) rectangle (\r*4-0.2,4)
		(\r*4+0.2,0) rectangle (\r*4+0.4,4);
		\draw[black] 		(\r*4,0) -- (\r*4,4);
	}
	
	\foreach \r in {1,...,4}{
		\draw[black!30, very thin] (0,\r-0.5) -- (12,\r-0.5);}
	\foreach \r in {1,...,20}{
		\draw[black!30, very thin] (\r-0.5,0) -- (\r-0.5,4);}
	
	\draw[red, thin] 	(0,0) -- (12,0);
	\draw[black]				(0,4) -- (12,4);
	
	\draw[red, very thin]	(0,0.03) -- (4.36,0.03) -- (4.36,1.9)--(8.2,2.1)--(8.2,0.03)--(12,0.03);
	\draw[red, very thin]	(0,0.06) -- (4.32,0.06) -- (4.32,3.83)--(8.24,3.83)--(8.24,0.06)--(12,0.06);
	\draw[red, very thin]	(0,0.09) -- (0.36,0.09) -- (0.36,1.9)--(4.28,2.1)--(4.28,3.86)--(8.28,3.86)--(8.28,0.09)--(12,0.09);
	\draw[red, very thin]	(0,0.12) -- (0.32,0.12) -- (0.32,3.89)--(8.32,3.89)--(8.32,0.12)--(12,0.12);
	\draw[red, very thin]	(0,0.15) -- (0.28,0.15) -- (0.28,3.92)--(8.36,3.92)--(8.36,1.9)--(11.68,2.1)--(11.68,0.15)--(12,0.15);
	\draw[red, very thin]	(0,0.18) -- (0.24,0.18) -- (0.18,3.95)--(11.72,3.95)--(11.72,0.18)--(12,0.18);
	
	\node[] at (6,3) {$C^0$};
	\node[] at (2,3) {$C^1$};
	\node[] at (10,3) {$C^2$};
\end{tikzpicture}
\end{center}

We now present the formal construction. We will simultaneously define the functions $f_1,\ldots,f_q$. %When writing out vectors concerning the function $f_i$ we will be doing so with respect to the ordered basis $\{e_i,\ldots,e_k,e_1,\ldots,e_{i-1}\}$. 
Denote by $\pi_i$ the projection of $\R^k$ onto $\langle e_i\rangle$ and by $\hat\pi_i$ the projection onto $e_i^\perp$. 

Set $a=\frac{1}{2k+2}$, $b_i=\frac{i}{2k+2}$ and $c_i=\frac{i+1}{2k+2}$.

We start by defining $f_i$ on the set $\pi_i^{-1}([0,r))$ which will have image $[0,1)$ and then extend by requiring that $f_i(v+kre_i)=f_i(v)+k$ for all $v\in\pi_i^{-1}([0,r))$ and all $k\in\Z$.

Enumerate the $\mathcal C_r$-cubes whose root lies on $e_i^\perp$ as $C^0,\ldots$. For $z\in [1-2^{-j},1-2^{-(j+1)})$ we define 
\[f^{-1}_i(z)=A^+_i(z)\cup A^-_i(z) \cup B_i(z) \cup C_i(z)\]
where 
\begin{itemize}
 \item $A^+(z) = (r-a+az)e_i + \setcon{v\in e_i^\perp}{d_\infty(v,C^0\cup\ldots\cup C^{j-1})\leq b_i+\frac{z}{2k+2}}$.
 \item $A^-(z) = aze_i+\setcon{v\in e_i^\perp}{d_\infty(v,C^0\cup\ldots\cup C^{j})\geq b_i+\frac{z}{2k+2}}$.
\end{itemize}

Now set $D_i(z)$ to be the closure of 
\[
    e_i^\perp \setminus \hat{\pi}_i(A^+(z)\cup A^-(z))
\] 
which is contained in the $\ell^\infty$ $\frac12$-neighbourhood of $C^j$.

Set $z'=2^{j+1}(z-(1-2^{-j}))$. As $z$ ranges from $1-2^{-j}$ to $1-2^{-(j+1)}$, $z'$ ranges from $0$ to $1$.

We define the hyperplane
\begin{equation}\label{eq:H_i}
     H_i=\setcon{\sum_{l=1}^k \lambda_le_l}{\lambda_i-\frac{1}{2(k-1)(r+1)}\sum_{l\neq i}\lambda_l=0}    
\end{equation}
We note that when $\sum_{l=1}^k \lambda_le_l\in H_i\cap [0,r+1]^k$, we have $\lambda_i\in[0,\frac12]$.

Let $v_j$ be the root of the $\mathcal C$-cube containing the root of the $\mathcal C_r$-cube $C^j$. We define
\[
    C'_i(z)= \hat\pi_i^{-1}(D_i(z))\cap (H_i+v_j+z'(r+1)e_i).
\]
Now to get $C_i(z)$ from $C'_i(z)$ we restrict the coefficient of $e_i$ to the interval $[az,r-a+az]$, i.e. for each $\sum_{l=1}^k \lambda_le_l\in C'_i(z)$ we add the following element to $C_i(z)$
\[
    \min\{\max\{\lambda_i,az\},r-a+az\}e_i + \sum_{l\neq i} \lambda_le_l
\]
By construction,
\[
    C_i(z)\subseteq \left(aze_i+e_i^\perp\right) \cup \left((r-a+az)e_i+e_i^\perp\right) \cup (H_i+z'(r+1)e_i)
\]
We note that for $z$ sufficiently close to $1-2^{-i}$, $C_i(z)\subseteq \left(aze_i+e_i^\perp\right)$ and for $z$ sufficiently close to $1-2^{-(i+1)}$, $C_i(z)\subseteq \left((r-a+az)e_i+e_i^\perp\right)$.

Finally, we define $B_i(z)$. For each $v\in e_i^\perp$, $\hat\pi_i^{-1}(v)$ intersects $A^+(z)\cup A^-(z)\cup C_i(z)$ in either one, two or three points, and it will always be one whenever $v$ is not in $\partial D_i(z)$. When this intersection contains more than one point, we add the minimal length line containing all the points in the intersection to $B_i(z)$. By construction, every line in $B_i(z)$ is parallel to $e_i$. More specifically, $B_i(z)$ is contained in a union of hyperplanes each of which is of the form
\[
 \lambda_le_l + e_l^\perp
\]
for some $l\neq i$, where each $\lambda_l\in r\Z\pm (b_i+\frac{z}{2k+2})$. Note that $\hat\pi_i(B_i)\subseteq \partial D_i(z)$.

By construction, for each $z$, $f_i^{-1}(z)$ is a closed set and $\R^k\setminus f_i^{-1}(z)$ has exactly two connected components, both of which are open.

We now verify that $f_i$ is a continuous function.

\textbf{Step 1: $f$ is well-defined:} We will prove that whenever $z<z'$, $f_i^{-1}(z)\cap f_i^{-1}(z')=\emptyset$. Specifically, if $v=\sum_{l=1}^k \lambda_l e_l \in f_i^{-1}(z)$ and $v'=\lambda'_ie_i + \sum_{l\neq i} \lambda_l e_l \in f_i^{-1}(z')$, then $\lambda_i<\lambda'_i$.

This is clear if there is some integer $m$ such that $z<m\leq z'$, therefore we may assume that $0<z<z'<1$. If $v\in A^-_i(z)$ then $\lambda_i=az$ and $\lambda'_i\geq az'>az$. Similarly, if $v'\in A^+_i(z')$, then $\lambda'_i=r-a+az'>r-a+az\geq \lambda_i$. 

We have $\hat\pi_i(A^+_i(z))\subseteq \hat\pi_i(A^+_i(z'))^\circ$ and $\hat\pi_i(A^-_i(z'))\subseteq \hat\pi_i(A^-_i(z))^\circ$. Also, $\hat\pi_i(B_i)\subseteq \partial D_i(z)$. Therefore, if $v\in A^+(z)\cup B_i(z)$, then $v'\in A^+(z')$ and if $v'\in A_i^-(z')\cup B_i(z')$ then $v\in A^-(z)$, so, the result holds in these cases as well. The remaining cases are $v\in C_i(z)\wedge v'\not\in A_i^+(z')$ and $v'\in C_i(z')\wedge v\not\in A_i^-(z)$, we will resolve the first of these, the second follows by a similar argument. As $v\in C_i(z)$ and $D_i(z)\subset D_i(z')^\circ$, $v'\not\in A^-_i(z')\cup B_i(z')$, so $v'\in C_i(z')$. Now as $v\in C_i(z)$ and $v'\in C_i(z')$, by construction, $\lambda_i<\lambda'_i$.

\textbf{Step 2: the domain of $f_i$ is $\R^k$:}
Let $v=\sum_{l=1}^k \lambda_le_l$. We may assume that $\lambda_i\in[0,r)$, as if $v\in f^{-1}(z)$ then $v+kre_i\in f^{-1}(z+k)$ for each $k\in \Z$. Given a subinterval $I\subset [0,r)$, define $V_i(I)=\setcon{\lambda e_i + \sum_{l\neq i} \lambda_le_l}{\lambda\in I}$ and set
\begin{align*}
    z^- & =\sup\setcon{z\in[0,1)]}{f^{-1}(z)\cap V_i[0,\lambda_i)\neq\emptyset} \\
    z^+ & =\inf\setcon{z\in[0,1)]}{f^{-1}(z)\cap V_i(\lambda_i,r)\neq\emptyset}
\end{align*}
By Step 1, for each $z$, $\R^k\setminus f^{-1}(z)$ is a union of two connected components, one of which contains $f^{-1}(-\infty,z)$, and the other contains $f^{-1}(z,+\infty)$. Thus $z^-\leq z^+$. For every $z\in [0,1)$, $f^{-1}(z)\cap V_i[0,r)\neq \emptyset$, so we must have $v\in f^{-1}(z)$ for every $z\in (z^-,z^+)$. As the sets $f^{-1}(z)$ are pairwise disjoint by Step 1, it follows that $z^-\geq z^+$. Set $z=z^-=z^+$.

If neither the supremum nor the infimum are attained in the above equations, then $v\in f^{-1}(z)$, as $f^{-1}(z)\cap V_i[0,r)\neq \emptyset$. If both the supremum and the infimum are attained, then $f^{-1}(z)$ contains two points 
\[
\lambda' e_i + \sum_{l\neq i} \lambda_le_l,\quad \lambda'' e_i + \sum_{l\neq i} \lambda_le_l
\]
with $\lambda'<\lambda<\lambda''$. By construction, the intersection of each $f^{-1}(z)$ with any line parallel to $e_i$ is connected (see the construction of the sets $B_i(z)$), so $v\in f^{-1}(z)$.

In the remaining cases, exactly one of the supremum or infimum is attained. We will look at the case where the supremum is attained, the other case is dealt with similarly. This means that for some $\lambda'<\lambda_i$,
\[
v'=\lambda' e_i + \sum_{l\neq i} \lambda_le_l\in f^{-1}(z).
\]
Firstly, we show that $v'\in B_i(z)$. Suppose this is not the case, then for all sufficiently small $\varepsilon>0$, there is some $\delta>0$ such that
\[
    v'+\varepsilon e_i \in f^{-1}(z+\delta).
\]
This contradicts the fact that $z=\sup\setcon{z'\in[0,1)]}{f^{-1}(z')\cap V_i[0,\lambda_i)\neq\emptyset}$.

Specifically, if $v'\in (A^-_i(z))^\circ$, then $\lambda'=az$, and $v'+\varepsilon e_i \in A^-_i(z+\frac{\varepsilon}{a})$ for all sufficiently small $\varepsilon>0$.
Similarly, if $v'\in A^+_i(z)$, then $v'+\varepsilon e_i \in A^+_i(z+\frac{\varepsilon}{a})$ for all sufficiently small $\varepsilon>0$.
The remaining case is that $v'\in C_i(z)$, in which case for all sufficiently small $\varepsilon>0$ there is some $\delta>0$ such that $v'+\varepsilon e_i \in C_i(z+\delta)$.

Now fix $v''=\lambda'' e_i + \sum_{l\neq i} \lambda_le_l\in B_i(z)$ so that $\lambda''$ is maximal. As $B_i(z)$ necessarily intersects $\pi_i^{-1}(\sum_{l\neq i} \lambda_le_l)$ in a closed finite length interval, this is possible. Note that, by construction $v''\in C_i(z)$. We claim that $v''=v$. Otherwise, as above, for all sufficiently small $\varepsilon>0$ there is some $\delta>0$ such that $v''+\varepsilon e_i \in C_i(z+\delta)$. As before, this contradicts the assumption that $z=\sup\setcon{z'\in[0,1)]}{f^{-1}(z')\cap V_i[0,\lambda_i)\neq\emptyset}$.

% Suppose $v\not\in \bigcup_{z\in [0,1)} A_i^+(z)\cup A^-_i(z)\cup B_i(z)$ and fix $j$ so that $v\in C^j$.

\textbf{Step 3: $f$ is continuous:} Let $(a,b)$ be an open interval in $\R$. By steps 1 and 2 $\R^k\setminus \left(f^{-1}(a)\cup f^{-1}(b)\right)$ is a union of $3$ connected components, $f_i^{-1}(-\infty,a)\cup f_i^{-1}(a,b)\cup f_i^{-1}(b,+\infty)$. As $f^{-1}(a)\cup f^{-1}(b)$ is a closed set, each of these connected components is open, so $f_i^{-1}(a,b)$ is open. Thus $f_i$ is continuous.
\medskip

\textbf{Completing the proof:} Let us bound $\cOv(f)$. For any $z=(z_1,\ldots,z_q)\in \R^q$, 
\[
f^{-1}(z)= \bigcap_{i=1}^q \left( A^+_i(z_i)\cup A^-_i(z_i) \cup B_i(z_i) \cup C_i(z_i)\right)
\]
Firstly, we note that $A^+_i(z_i)\cup A^-_i(z_i) \cup C_i(z_i)$ is contained in
\[
    (az_ie_i+e_i^\perp) \cup (r-a+az_ie_i+e_i^\perp) \cup C'_i(z_i)
\]
Also, $B_i(z_i)$ is contained in $\bigcup_{j\neq i} B^j_i(z_i)$ where 
\[
  B^j_i(z_i)=\bigcup_{\lambda_j\in r\Z\pm(b_i+\frac{z_i}{2k+2}} \lambda_je_j+e_j^{\perp}.
\]
It follows that for every $i\neq j$, $B^j_i\cap (A^+_j(z_j)\cup A^-_j(z_j))=\emptyset$, since $az_j,r-a+az_j\not\in r\Z\pm(b_i+\frac{z_i}{2k+2})$. Using similar reasoning, for any $i\neq i'$, $B^j_i\cap B^j_{i'}=\emptyset$, since $|b_i-b_{i'}|\geq \frac{1}{2k+2} > \frac{z_i}{2k+2}$.

As a result, $f^{-1}(z)$ is contained in a union of countably many intersections of $q$ translates of hyperplanes in general position (i.e.\ translates of $(k-q)$-dimensional subspaces of $\R^k$). Those that do not include an intersection with any $C'_i(z_i)$ are completely contained in $Y^{k-q}$, as at least $q$ of their coordinates are in $r\Z+(-\frac12,\frac12)$. Now each $C_i(z_i)$ is contained in the pixellation of some $\mathcal{C}_r$-cube (denoted $P_i$), and there is a uniform bound on the number of $\lambda_je_j+e_j^{\perp}$ with $\lambda_j\in r\Z\pm(b_m+\frac{z_m}{2k+2})$ which intersect $P_i$. Thus $f^{-1}(z)\setminus Y^{k-l}$ is contained in the intersection of $P_i$ with some uniformly bounded number (say bounded from above by $N$) of translates of $(k-q)$-dimensional subspaces of $\R^k$. Each of these subspaces can be covered by a union of at most $N'r^{k-q}$ $\mathcal C$-cubes, so
\[
 \Cov^c(f^{-1}(z)) \leq \Cov^c(Z\cap Y^{k-q})+NN'r^{k-q}
\]
as required.

\subsection{Products and Lipschitz Euclidean fibrations}
We next present some elementary upper bounds for direct products and extensions. Our first goal is to prove Theorem \ref{thm:productub} which we recall for convenience.

\begin{theorem*}  Let $X,Y$ be metric spaces with bounded geometry. For any $q,q'\geq 1$ we have
\[
 \TO^{q+q'}_{X\times Y}(r) \lesssim \TO^q_X(r)\TO^{q'}_Y(r).
\]
\end{theorem*}
\begin{proof}
Choose $Z\subseteq X\times Y$ with $\Cov^1(Z)\leq r$. For convenience we consider the $\ell^\infty$ product metric on $X\times Y$.\footnote{Changing between $\ell^p$ product metrics defines homeomorphisms satisfying the hypotheses of Proposition \ref{prop:cTOmonotone} so does not alter $\TO$ up to $\simeq$.}  Since projections are $1$-Lipschitz, $\Cov^1(\pi_X(Z)\leq r$ and $\Cov^1(\pi_Y(Z))\leq r$.

Choose $f\in C(\pi_X(Z),\R^q)$ and $g\in C(\pi_Y(Z),\R^{q'})$ with $\Ov(f)\leq\sTO^q_X(r)$ and $\Ov(g)\leq\sTO^{q'}_X(r)$. Now define
\[
 \phi:Z\to\R^{q+q'} \quad \textup{by}\quad \phi(x,y)=(f(x),g(y)).
\]
Let $(z,z')\in\R^q\times\R^{q'}$. We have $\phi^{-1}(z,z')\subseteq f^{-1}(z)\times g^{-1}(z')$. Since $f^{-1}(z)$ can be covered by at most $\TO^q_X(r)$ balls of radius $1$ (centred at $x_1,\ldots,x_k$) and $g^{-1}(z')$ covered by at most $\TO^{q'}_Y(r)$ balls of radius $1$ (centred at $y_1,\ldots,y_l$), we see that $\phi^{-1}(z,z')$ can be covered by at most $\TO^q_X(r)\TO^{q'}_Y(r)$ $(\ell^\infty-)$balls of radius $1$ centred at $(x_i,y_j)$.
\end{proof}

From Theorem \ref{thm:productub} we quickly deduce Corollary \ref{cor:prodH2Tub} which we again recall for convenience.

\begin{corollary*}
    Let $X=(\HH^2)^k\times (T_3)^l$, where $T_3$ is the infinite $3$-regular tree. We have
    \[
        \TO^{k+l}_X(r) \lesssim \ln(1+r)^{2k+l}
    \]
\end{corollary*}
\begin{proof}
    By Corollary \ref{cor:fgnilpTO1}, $\TO^1_{\HH^2}(r)\lesssim \ln(1+r)^2$. Combining Theorem \ref{thm:compTO1cutwidth} and \cite{CFST-treecutwidth}, we have $\TO^1_{T_3}(r)\simeq \ln(1+r)$. Now, by Theorem \ref{thm:productub},
    \[
        \TO^{k+l}_X(r) \lesssim \left(\TO^1_{\HH^2}(r)\right)^k\left(\TO^1_{T_3}(r)\right)^l \lesssim \ln(1+r)^{2k+l}
    \]
    as required.
\end{proof}

One other quick consequence is

\begin{corollary} $\TO^2_{\SOL}(r) \lesssim \ln(1+r)^4$.
\end{corollary}
\begin{proof} There is a continuous biLipschitz embedding $\SOL\to\HH^2\times\HH^2$, so by Proposition \ref{prop:TOmonotone} and Theorem \ref{thm:productub}
\[
 \TO^2_{\SOL}(r) \lesssim \TO^1_{\HH^2}(r)^2 \lesssim \ln(1+r)^4.
\]
\end{proof}

We can also find bounds on a more general class of ``product-like'' metric spaces, under the additional assumption that one of the factors topologically embeds into some $\R^k$. 

\begin{proposition}\label{prop:semidirprod2} Let $(A,d_A)$ be a metric space and let $(B,d_B)$ be a metric space which admits a continuous injective map $g:B\to\R^k$ for some $k$. Let $d$ be a metric on the set $A\times B$ which satisfies
\begin{equation}\label{eq:binvariant}
 d((a,b),(a',b))=d_A(a,a')
\end{equation}
for all $a,a'\in A$ and all $b\in B$.

If there is some constant $C$ such that $\Cov^1(\pi_A(Z),d_A)\leq C\Cov^1(Z,d)$ for all $Z\subseteq A\times B$ with $\Cov^1(Z)<+\infty$, where $\pi_A$ denotes projection onto the $A$-factor,
 then for all $q\in\N$
\[
    \TO^{k+q}_{(A\times B,d)}(r) \leq \TO^q_{(A,d_A)}(Cr).
\]
\end{proposition}
\begin{proof} Let $Z\subseteq A\times B$ with $\Cov^1(Z)\leq r$. Fix a function $f:\pi_A(Z)\to \R^q$ with $\Ov(f)\leq \TO^q_{(A,d_A)}(Cr)$. We now define
\[
 \phi:Z\to \R^{q}\times\R^{k} \quad \textup{by} \quad \phi(a,b)=(f(a),g(b))
\]
Let $(x,y)\in\R^{q}\times\R^{k}$. We have
\[
\phi^{-1}(x,y)\subseteq f^{-1}(x) \times\{g^{-1}(y)\}
\]
If $y\not\in\im(g)$ there is nothing to prove. Otherwise, by \eqref{eq:binvariant}, $(f^{-1}(x) \times\{g^{-1}(y)\},d)$ is isometric to $(f^{-1}(x),d_A)$, so $\Ov(\phi)= \Ov(f)$. As this holds for all $Z\subseteq A\times B$ with $\Cov^1(Z)\leq r$, we have
\[
 \TO^{k+q}_{(A\times B,d)}(r) \lesssim \TO^q_{(A,d_A)}(C
 r).\qedhere
\]
\end{proof}

The key application of this proposition is to prove Corollary \ref{cor:ses} which we recall for convenience.

\begin{corollary*} Let $1\to N \to G \to_{\pi} H \to 1$ be a (continuous) short exact sequence of compactly generated locally compact groups, and let $d_G$ be a proper left-invariant metric on $G$. If there is a continuous injective map $N\to \R^k$ for some $k$, then for all $q\in\N$, we have
\[
\TO^{k+q}_{(G,d_G)}(r) \lesssim \TO^q_{(H,d_H)}(r)
\]
where $d_H$ is the proper left-invariant metric $d_H(h,h')=d_G(\pi^{-1}(h),\pi^{-1}(h'))$.
\end{corollary*}

\begin{proof} As a set $G=H\times N$. By construction, the map $\pi$ is $1$-Lipschitz, so $\Cov^1(\pi(Z))\leq \Cov^1(Z)$ for all $Z\subseteq G$ with $\Cov^1(Z)<+\infty$. The result then follows from Proposition \ref{prop:semidirprod}.
\end{proof}

To give one application of this, we find upper bounds for connected simply-connected nilpotent Lie groups. Let $N$ be a connected simply-connected nilpotent Lie group, let $N_1=N$ and $N_{i+1}=[N,N_i]$ be the terms of the lower central series. We have $C^i/C^{i+1}\cong\R^{k_i}$. Set $k=\sum k_i=\asdim(N)$. Let $a^N=(a^N_1,\ldots,a^N_k)$ be the tuple consisting of $k_1$ $1$s followed by $k_2$ $2$s, and so on.

\begin{corollary} Let $G$ be a simply connected nilpotent Lie group. For $2\leq q \leq k$, we have
\[
    \TO^q_N(r) \lesssim r^{1-\frac{1}{\sum_{i=1}{k-(q-1)} a^N_i}}.
\]    
\end{corollary}
\begin{proof}
    We note that $N=N_1$ contains a central subgroup $C$ isomorphic to $\R$ such that $N_2=N_1/C$ is simply connected and has the associated tuple $a^{N'}=(a^N_1,\ldots,a^N_{k-1})$. From Corollary \ref{cor:ses}, we deduce that $\TO^q_N(r) \lesssim \TO^{q-1}_{N_2}$. Repeating this, we obtain
    \[
        \TO^q_N(r)\lesssim \TO^1_{N_q}(r) 
    \]
    where $N_q$ has the associated tuple $a^{N_q}=(a^N_1,\ldots,a^N_{k-(q-1)})$. By \cite[Theorem 8.1]{HumeMackTess-Pprof},
    \[
        \TO^1_{N_q}(r) \simeq r^{1-\frac{1}{\sum_{i=1}^{k-(q-1)} a^N_i}}
    \]
    as required.
\end{proof}
This is not sharp, even for Euclidean spaces.

\subsection{Real hyperbolic spaces}\label{sec:hypTO}
We now present results for real hyperbolic spaces $\HH^d$, we begin with lower bounds.

\begin{proposition}\label{prop:hyplbpower} We have 
\[
\TO^q_{\HH^{d}}(r)\gtrsim \left\{\begin{array}{cc}
    r^{1-q/(d-1)} & q<d-1 \\
    \ln(1+r) & q=d-1
\end{array} \right.
\]
\end{proposition}
\begin{proof}
Suppose $q<d-1$. For each $t$ let $S_t$ be a sphere of radius $t$ in $\HH^d$. Using standard results in hyperbolic geometry, there is a homeomorphism $f$ from $S_t$ to the Euclidean sphere of dimension $d-1$ and radius $\sinh(t)$ and a constant $K=K(d)$ such that for any $A\subseteq S_t$,
\begin{equation}\label{eq:Covcomp}
K^{-1}\Cov^1(f(A))\leq \Cov^1(A) \leq K\Cov^1(f(A)). 
\end{equation}
By Corollary \ref{cor:Gromov_waist} for any $g\in C(S_t,\R^q)$ there is a point $z\in \R^q$ such that $\Cov^1(f^{-1}(g^{-1}(z)))\geq \beta_{d-1,q} \sinh(t)^{d-1-q}$. Thus, by $\eqref{eq:Covcomp}$, 
\[
\Ov(g)\geq K^{-1}\beta_{d-1,q} \sinh(t)^{d-1-q} \geq k \Cov^1(S_t)^{\frac{d-1-q}{d-1}}.
\]
As this holds for every $g$ and every $t$, we have
\[
 \TO^q_{\HH^d}(r) \gtrsim r^{1-q/(d-1)}.
\]
For $q=d-1$, by \cite[Lemma I.1.13]{BH99} the map $\Phi:B(0;R)_{\HH^{d}}\to B(0;R)_{\R^{d}}$ given by $\Phi(r,\theta)=(2\tanh^{-1}(r),\theta)$ is $1$-Lipschitz. Therefore for any set $A\subset \R^{d+1}$ with $\Cov^1(A)=\ell$ we have $\Cov^1(\Phi^{-1}(A))\geq \ell$. 

Let $f:B(R)_{\HH^{d+1}}\to\R^d$ be continuous and define $g=f\circ \Phi$. By the waist inequality above, there is some $z\in\R^d$ such that $\Cov^1(g^{-1}(z))\geq cR$, so $\Cov^1(f^{-1}(z))=\Cov^1(\Phi(g^{-1}(z)))\geq cR$.
\end{proof}

Next, we move to upper bounds.

\begin{proposition}\label{prop:hypdimdrop} For all $n\geq 2$ and $q\geq 1$, we have $\TO_{\HH^{n+1}}^{q+1}(r) \lesssim \TO_{\HH^{n}}^{q}(r)$.
\end{proposition}
\begin{proof} We fix a homeomorphism $\phi:\HH^{n+1}\to\HH^n\times\R$ given (in the halfspace model) by
\[
 \phi(x_1,\ldots,x_n;y))=((x_1,\ldots,x_{n-1};y),x_n).
\]
We can then apply Proposition \ref{prop:semidirprod2} with $A=\HH^n$, $B=\R$ and the usual metrics $d=d_{\HH^{n+1}}$ and $d_A=d_{\HH^{n}}$.
\end{proof}

\begin{corollary} We have 
\[
\TO^q_{\HH^{d}}(r)\lesssim \left\{\begin{array}{cc}
    r^{1-1/(d-q)} & q<d-1 \\
    \ln(1+r)^2 & q=d-1
\end{array} \right.
\]
\end{corollary}
\begin{proof}
By repeated use of Proposition \ref{prop:hypdimdrop}
\[
 \TO^{q}_{\HH^d}(r) \lesssim \TO^{q-1}_{\HH^{d-1}}(r) \lesssim \ldots \lesssim \TO^1_{\HH^{d-q+1}}(r).
\]
We have $\TO^1_{\HH^{d-q+1}}(r)\simeq \sTO^1_{X^{d-q+1}}(r)\simeq\cw_{(X^{d-q+1})^{\leq 1}}(r)$ by Theorems \ref{thm:defnequiv} and \ref{thm:compTO1cutwidth} for a suitable choice of simplicial complex $X^{d-q+1}$, for instance the universal cover of a finite triangulation of a closed hyperbolic $(d-q+1)$-manifold. Finally, by \cite[Corollary 1.7]{HHKL23},
\[
 \cw_{(X^{d-q+1})^{\leq 1}}(r) \lesssim \left\{\begin{array}{cc}
    r^{1-1/(d-q)} & q<d-1 \\
    \ln(1+r)^2 & q=d-1
\end{array} \right.\qedhere
\] 
\end{proof}

\subsection{Symmetric spaces}
Here we give the proof of Theorem \ref{cor:ubsymspace}, which we recall for convenience.

\begin{corollary*} Let $X$ be a Riemannian symmetric space of dimension $n$ and rank $d$. For every $q\geq 1$ we have
    \[
        \TO^{n-d+q}_X(r)\lesssim \TO^q_{\R^d}(r)\simeq r^{1-\frac{q}{d}}.
    \]
\end{corollary*}

\begin{proof}[Proof of Theorem \ref{cor:ubsymspace}]
Let $X= G/K$ be a Riemannian symmetric space. Recall that the Iwasawa decomposition states that the multiplication map $K\times A \times N \to G$ is a homeomorphism, where $A$ is isometric to $\R^d$ and $N$ is a simply connected nilpotent Lie group. In particular, this means that $N$ is diffeomorphic to Euclidean space of the same manifold dimension $\dim(N)=n-d$ \cite[Theorem 1.127]{Knapp-beyond}. Moreover, there is a proper left-invariant metric $d$ on $AN$ such that $(AN,d)$ is isometric to $X$.  Thus, we have a short exact sequence
\[
 1 \to N \to AN \to A \to 1
\]
which satisfies the hypotheses of Corollary \ref{cor:ses}. Therefore,
\[
 \TO^{n-d+q}_X(r) \simeq \TO^{\dim(N)+q}_{(AN,d)}(r) \lesssim \TO^q_A(r)\simeq r^{1-q/d}.\qedhere
\]
\end{proof}

\subsection{Assouad-Nagata dimension}
Here we prove Corollary \ref{cor:TO1finANdim}, which again we restate for convenience.

\begin{corollary*} Let $X\in\bS$ have finite Assouad-Nagata dimension. Then, for every $q\geq 1$,
\[
 \sTO^q_X(r) \lesssim \frac{r}{\ln(1+r)}.
\]
\end{corollary*}
\begin{proof} By \cite[Corollary 4.3]{HHKL23} and \cite[Theorem 1.5]{HumSepExp},
\[
 \sTO^1_X(r) \lesssim \frac{r}{\ln(1+r)}.
\]
Applying Theorem \ref{thm:TOprops}(ii) $(q-1)$ times, we see that
\[
 \sTO^q_X(r) \lesssim \frac{r}{\ln(1+r)}.\qedhere
\]
\end{proof}

\section{Coarse constructions}\label{sec:coarsecon}
In this section we introduce coarse constructions as a generalisation of the coarse wirings from \cite{BarrettHume}.

\begin{definition} Let $Z,Y\in\bS$ with $|Z|<+\infty$. A continuous map $f\in C(Z,Y)$ is called a \textbf{$k$-coarse construction} if
\begin{itemize}
 \item for each $d$, $f(Z^{\leq d})\subseteq Y^{\leq d}$, and
 \item $\sup_{\sigma\in \CY}\abs{\setcon{\sigma'\in\CZ}{\sigma\cap f(\sigma')\neq\emptyset}}\leq k$.
\end{itemize}
A simple example of a $1$-coarse construction is the natural homeomorphism from a simplicial complex to any geometric subdivision of it.

The \textbf{volume} of a coarse construction $f$, $\vol(f)$ is the minimal number of $0$-simplices in a subcomplex of $Y$ which contains the image of $f$. We define
$\con^k_{Z\to Y}$ to be the minimal volume of a $k$-coarse construction $f\in C(Z,Y)$ or $+\infty$ if no such map exists.

Given $X,Y\in\bS$ we define
\[
 \con^k_{X\to Y}(n)=\max\setcon{\con^k(Z\to Y)}{Z\leq X,\ |Z|\leq n}.
\]
\end{definition}

Coarse constructions satisfy a natural composition law.

\begin{lemma} Let $X,Y,Z\in\bS$. We have
\[
\con^{kl}_{X\to Z}(n) \leq \con^l_{Y\to Z}\left(\con^k_{X\to Y}(n)\right).
\]
\end{lemma}
\begin{proof} Fix $n$.  If $\con^l_{Y\to Z}\left(\con^k_{X\to Y}(n)\right)=+\infty$ then there is nothing to prove, so assume it is finite. Let $A\leq X$ with $\abs{A^0}\leq n$. Then there exists a
  coarse $k$-construction $\psi$ of $A$ into $Y$ contained in a subcomplex $B\leq Y$ with $|B^0|\leq \con^k_{X\to Y}(n)$,  and a coarse $l$-construction $\psi'$ of $B$ into $Z$ contained in a subcomplex $C\leq Z$
  with
  $|C^0| \leq \con^l_{Y\to Z}\left(\con^k_{X\to Y}(n)\right)$.

Now set $\psi''=\psi'\circ\psi:A\to Z$. It is clear that $\psi''\in C(A,Z)$ and $\psi''(A^{\leq d})\subseteq Z^{\leq d}$ for all $d$. Now let $\sigma\in \CZ$. There are at most $l$ simplices $\tau_1,\ldots,\tau_l$ in $\CY$ such that $\psi'(\tau_i)\cap \sigma\neq\emptyset$. For each $\tau_i$ there are at most $k$ simplices $\sigma_{i,1},\ldots,\sigma_{i,k}$ simplices in $X$ such that $\psi( \sigma_{i,j})\cap \tau_i\neq\emptyset$. It follows that
\[
\sup_{\sigma\in \CZ}\abs{\setcon{\sigma'\in\CX}{\sigma\cap \psi''(\sigma')\neq\emptyset}}\leq kl.
\]
Thus, $\psi''$ is a $kl$-coarse construction.
\end{proof}

%\begin{example} Let $Z\in\bS$ be finite. The natural map from $Z$ to its barycentric subdivision is a ??? $\Delta_Z$-coarse construction.
%\end{example}

%
%\
%In analogy with \cite{BarrettHume} we define coarse fillings of simplicial complexes into metric spaces.
%
%\begin{definition} Let $S,T\in\BSC$ with $S$ finite. A continuous map $f:S\to T$ with the property that $f(S^{\leq k})\subseteq T^{\leq k}$ for all $k\leq \dim(S)$ is called a \textbf{$K$-coarse construction} (of $S$ in $T$) if, for all simplices $\sigma'\subset T$,
%	\begin{equation}\label{eq:Kcoarsepullback}
%		\setcon{\sigma\subset S}{\sigma\cap f^{-1}(\sigma')\neq\emptyset}\leq K.
%	\end{equation}
%The \textbf{volume} of $f$ is $\vol(f)=|f(S)^{(0)}|$.
%
%We denote by $\vol^K(S\to T)$ the minimal volume of a $K$-coarse construction of $S$ in $T$.
%\end{definition}
%
%\begin{definition}
%	Let $X,Y\in\BSC$ and let $d\in\N$. The [insert name here] of $X$ into $Y$ is defined by
%	\[
%	 {}^d\con^K_{X\into Y}(r) = \max\setcon{\vol^K(S\to Y)}{S\leq X^{\leq d},\ |S|\leq r}.
%	\]
%\end{definition}

\subsection{Lower bounds on fillings from topological overlap}
The goal of this subsection is to prove Theorem \ref{thm:TOcon}.

\begin{theorem*} Let $X,Y\in\bS$ and suppose $\con^k_{X\to Y}$ takes finite values. For each $d\geq 1$,
\[
 \sTO^q_X(n) \leq k \sTO^q_Y(\con^k_{X\to Y}(n))
\]
\end{theorem*}
\begin{proof} Choose $Z\leq X$ with $|Z|\leq n$ and $\phi:Z\to Y$ be a $k$-coarse construction whose image $Z'$ has at most $\con^k_{X\to Y}(n)$ $0$-simplices.

Let $g:Z'\to\R^q$ be a continuous map satisfying
\[
 \sOv(g) \leq \sTO^q_Y(\con^k_{X\to Y}(n)).
\]
Now define $f=g\circ\phi\in C(Z,\R^q)$. We will bound $\sOv(f)$. Fix $z\in\R^q$ and let $\mathcal C'$ be the collection of all closed simplices $\sigma'$ in $Z'$ such that $z\in g(\sigma)$. Let $\sigma\in C\Gamma$. If $z\in f(\sigma)$ then $\phi(\sigma)$ has non-trivial intersection with some $\sigma'\in \mathcal C'$. By definition, there are at most $k$ such $\sigma$ for each element of $\mathcal C'$. Thus
\[
 \sTO^q(Z)\leq \sOv(f) \leq k\TO^q_Y(\con^k_{X\to Y}(n))
\]
As $Z$ was chosen arbitrarily, we have
\[
 \sTO^q_X(n) \leq k \sTO^q_Y(\con^k_{X\to Y}(n))
\]
as required.
\end{proof}

As a consequence we deduce Theorem \ref{thm:TOlbcon}
\begin{proof}
    Choose $X\in\bS$ so that $\TO^q_X(r)\simeq r$. Applying Theorem \ref{thm:TOcon} to this choice of $X$ we see that there is a constant $C>0$ such that
    \[
        \sTO^q_Y(\con^k_{X\to Y}(n))\gtrsim n.
    \]
    In particular, if $\con^k_{X\to Y}(n)\lesssim n^a\ln(1+n)^b$, then
    \[
        \sTO^q_Y(n^a\ln(1+n)^b) \gtrsim n,
    \]
    so $\sTO^q_Y(n)\gtrsim n^{\frac1a}\ln(1+n)^{\frac{-b}{a}}$.
\end{proof}

\subsection{Coarse constructions and topological overlap of horocyclic products of trees}

We give the definition of a horocyclic product of trees.

\begin{definition}
	Let $T$ denote the infinite $3$-regular tree, fix some $v\in VT$ and $\eta\in\partial_\infty T$. For each $x$ let $\gamma_x$ denote the unique geodesic ray starting at $x$ with boundary point $\eta$. We define
	\[
	 h_{v,\eta}(x)= d(x,\gamma_v\cap\gamma_x) - d(v,\gamma_x\cap\gamma_v).
	\]
	The horocyclic product of ${d+1}$ trees $H^{d+1}(T)$ is the flag complex with vertex set
	\[
	 \setcon{(v_0,\ldots,v_d)\in (VT)^{n}}{\sum_{i=0}^d h_{v,\eta}(v_i)=0}
	\]
	and $1$-simplices $(v_0,\ldots,v_d)(v'_0,\ldots,v'_d)$ whenever there exist $0\leq i <j\leq d$ such that $v_iv'_i,v_jv_j'\in ET$ and $v_k=v'_k$ for all $k\neq i,j$.
\end{definition}
We have $\dim(H^{d+1}_T)=d$ and $\deg(H^{d+1}_T)\leq 2d(d+1)$.

\begin{theorem}
	For every $d\geq 1$ and every $\Delta\in\N$, there is some $K=K(d,\Delta)$ such that every finite $Z\in \bS$ with $\dim(Z)\leq d$ and $\deg(Z)\leq \Delta$ admits a $K$-coarse construction into $H^{d+1}_T$ with volume at most $K|Z|\log_2^d(1+|Z|)$.
\end{theorem}

Before beginning the proof, we quickly sketch the idea. The image of the wiring will be contained in a subset of $H^{d+1}_T$ whose $0$-simplices can be identified with the set of $(d+1)$-tuples of binary sequences whose lengths sum to $\ell(d+1)$ where $\ell=\lceil\log_2(1+|Z|)\rceil$.  This $\ell$ is chosen so that there is an injection $\iota$ of $Z^0$ into binary sequences of length $\ell$. We then build a coarse construction of $Z$ into $H^{d+1}_T$ with the following property. Whenever a $0$-simplex $v=(\omega_0,\ldots,\omega_d)$ is contained in the image of some $\sigma\in\CZ$ then for every $i$ such that $|\omega_i|\geq \ell$ (of which there must be at least one) the initial subword of $\omega_i$ of length $\ell$ is the image under $\iota$ of one of the $0$-simplices of $\sigma$.

We build the coarse construction in two stages. The first stage is a subdivision $Z'$ of $Z$ with the property that every $k$-simplex in $Z$ is subdivided into $\sim \log_2^d(1+|Z|)^k$ simplices in $Z$. In the second stage we construct a simplicial map $Z'\to H^{d+1}_T$. We then prove that the composition of these maps is a $K$-coarse construction for a suitable $K$.

\begin{proof} 

\textbf{Step 1: Subdividing $Z$}

By definition, $Z$ is the topological realisation of some abstract simplicial complex $(S,\mathcal S)$. We enumerate $S=\{1,\ldots,n\}$ and let $\mathbb{S}$ be the set of all strictly increasing chains of elements of $\mathcal S$. We have that $(\mathcal S,\mathbb{S})$ is the barycentric subdivision of $(S,\mathcal S)$. Now take such a $(\mathcal S,\mathbb{S})$ whose topological realisation is $Z''$.

Fix $\ell\in\N$ so that $2^{\ell-1}\leq n \leq 2^\ell$ and define $Z'$ to be (the topological realisation of) the abstract simplicial complex on the set
\[
 D_\ell(\mathcal S)=\setcon{f:\mathcal S\to\N\cup 0}{\supp(f)\in \mathbb{S},\ \sum_{z\in\mathcal S}f(z)=(d+1)\ell}.
\]
where a family of distinct functions $f_0,\ldots,f_k$ span a $k$-simplex if and only if 
\begin{itemize}
\item $\bigcup_{i=0}^k\supp(f_i)\in\mathbb{S}$,
\item for each $i\neq j$, there exist $z',z''\in\mathcal S$ such that
\[
    f_i(z)=\left\{\begin{array}{cc}
        f_j(z)+1 & z=z'  \\
        f_j(z)-1 & z=z''  \\
        f_j(z) & z\neq z',z''
    \end{array}\right.
\]
\end{itemize}

One can equivalently define $Z'$ as a geometric subdivision of $Z''$ where we add new $0$-simplices at every point where all coordinates are integer multiples of $1/(d+1)\ell$, and new $k$-simplices for every set of $0$-simplices which pairwise agree on all but two coordinates and differ by exactly $1/(d+1)\ell$ in the other two. The geometric interpretation is helpful for constructing a homeomorphism $Z\to Z''$, the more combinatorial interpretation introduced above is helpful for the rest of the proof.

%
%\begin{center}
%\input{tikzsubdivision}
%$Z$, $Z''$ and $Z'$ in the non-existent situation where $(d+1)\ell=2$.
%\end{center}

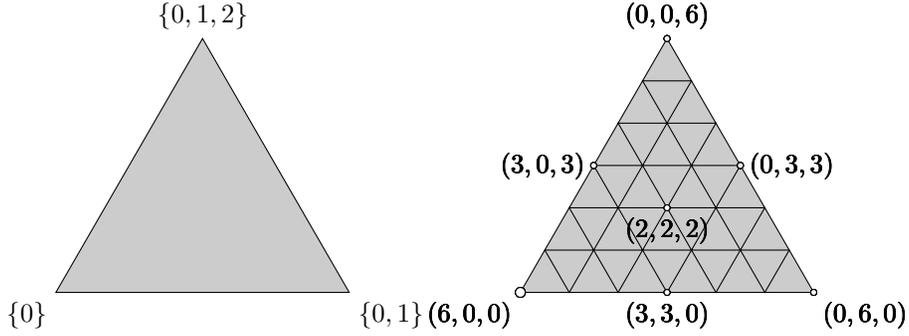
\begin{center}
\begin{tikzpicture}[scale=0.65]
	% Define the size of the triangles
	\pgfmathsetmacro{\triangleSize}{1}
	\pgfmathsetmacro{\triangleHeight}{\triangleSize*sqrt(3)/2}
	
	\draw[fill=black!20] (-10,0) -- ++(6*\triangleSize,0) -- ++(-3*\triangleSize,6*\triangleHeight) -- cycle;
	
	\node[below left] at (-10,0) {$\{0\}$};
	\node[below right] at (-4,0) {$\{0,1\}$};
	\node[above] at (-7,6*\triangleHeight) {$\{0,1,2\}$};
	
	\fill[black!20] (-\triangleSize/2,0) -- ++(6*\triangleSize,0) -- ++(-3*\triangleSize,6*\triangleHeight) -- cycle;    
	
	% Draw the row of triangles
	\foreach \i in {1,2,3,4,5,6} {
		\foreach \j in {1,...,\i} {
			% Calculate the x-coordinate of the top vertex of the current triangle
			\pgfmathsetmacro{\x}{(\i-1) * \triangleSize - \j*\triangleSize/2}
			
			% Calculate the y-coordinate of the top vertex of the current triangle
			\pgfmathsetmacro{\y}{(\j-1)*\triangleHeight}
			
			% Draw the equilateral triangle
			\draw (\x,\y) -- ++(\triangleSize,0) -- ++(-\triangleSize/2,\triangleHeight) -- cycle;
		}
		
		\draw[fill=white] (-0.5,0) circle (0.1);
		\node[below left] at (-0.5,0) {$(6,0,0)$};
		\draw[fill=white] (1,3*\triangleHeight) circle (0.06);    		
		\node[left] at (1,3*\triangleHeight) {$(3,0,3)$};
		\draw[fill=white] (4,3*\triangleHeight) circle (0.06);    
		\node[right] at (4,3*\triangleHeight) {$(0,3,3)$};
		\draw[fill=white] (5.5,0) circle (0.06);    
		\node[below right] at (5.5,0) {$(0,6,0)$};
		\draw[fill=white] (2.5,6*\triangleHeight) circle (0.06);    
		\node[above] at (2.5,6*\triangleHeight) {$(0,0,6)$};
		\draw[fill=white] (2.5,0) circle (0.06);    
		\node[below] at (2.5,0) {$(3,3,0)$};
		\draw[fill=white] (2.5,2*\triangleHeight) circle (0.06);    
		\node[below] at (2.5,2*\triangleHeight) {$(2,2,2)$};
	}
\end{tikzpicture}
The simplices in $Z'$ whose support is contained in the $2$-simplex $(\{0\},\{0,1\},\{0,1,2\})$ in $Z''$ in the case $\ell=d=2$. \\ Each function $f$ is written out as $(f(\{0\}),f(\{0,1\}),f(\{0,1,2\}))$.
\end{center}

\begin{lemma}\label{lem:boundZ'} We have $\dim(Z')=\dim(Z)=d$, $\deg(Z')\leq\Delta2^{\Delta}$ and $|Z'|\leq C|Z|\ell^d$ for some $C$ which depends only on $\Delta$.
\end{lemma}
\begin{proof} 
Suppose that $f_0,\ldots,f_{k}\in D_\ell(\mathcal S)$ are distinct and span a $k$-simplex. We will show that there is a well-defined injective map $\bigcup_{i=0}^k\supp(f_i)\to\{0,\ldots,d\}$, which assigns to each $z$ the unique $j$ such that $i\mapsto f_i(z)$ is constant on $\{0,\ldots,d\}\setminus\{j\}$.

For each $z\in\bigcup_{i=0}^k\supp(f_i)$, $\setcon{f_i(z)}{0\leq i \leq k}$ has at most $2$ elements, otherwise condition (iii) fails. Next, suppose for a contradiction that there exist $i,i',j,j'$ all distinct such that $f_{i}(z)=f_{i'}(z)< f_{j}(z)=f_{j'}(z)$, then there exist $z',z''$ such that $f_i(z')>f_j(z')$ and $f_{i'}(z'')>f_{j'}(z'')$. If $z'\neq z''$, then $f_i$ and $f_{j'}$ differ at $z,z'$ and $z''$, contradicting $(ii)$, so by $(iii)$, $f_i(z')=f_{i'}(z')$. Similarly, if $f_i(z''')\neq f_{i'}(z''')$ for some $z'''\neq z,z'$, then one of $f_i$ or $f_{i'}$ must differ with $f_j$ at $z,z',z'''$, which is a contradiction. Thus $f_i=f_{i'}$ which contradicts the assumption that the $f_i$ were distinct.
\medskip

Now let $f\in D_{\ell}(\mathcal S)$ and set $s=|\supp(f)|$. If $f,f'$ span a $1$-simplex and $\supp(f')\subseteq \supp(f)$, by (ii) and (iii) there are at most $s(s-1)$ possibilities for $f'$ (increase the function by 1 on one element of the support and decrease by one on another). Otherwise, there is some $z\in\supp(f)$ such that $f'(z)=f(z)-1$ and some $z'\not\in\supp(f)$ such that $f'(z')=1$. There are $s$ possibilities for $z$. By (i) $\supp(f)\cup\{z'\}\in\mathbb S$, so $\{z,z'\}\in\mathbb{S}$, so once $z$ has been chosen there are at most $\deg(Z'')\leq 2^{\Delta}$ possibilities for $Z'$. Thus 
\[
 \deg(Z')\leq \max\{(d+1)d, d2^{\Delta}\}\leq \Delta2^{\Delta}
\] 
since $s\leq d+1\leq\Delta$.
\medskip

Finally, $|Z'|=|D_\ell(\mathcal S)|\leq |\mathbb{S}| F_{d+1}((d+1)\ell)$ where $F_{k}(n)$ is the number of ways of writing $n$ as the ordered sum of $k$ non-negative integers (eg.\ $(0,2)$ and $(2,0)$ are considered as different elements of $F_2(2)$). An easy upper bound on $F_{k}(n)$ which is sufficient for our purposes is $(n+1)^{k-1}$, which is obtained by noting that the $k$th non-negative integer can always be determined once the previous $k-1$ are known.

Applying Fact \ref{fact:simplicebound} twice, we see that
\begin{align*}
	|Z'| 	&	\leq |\mathbb{S}| ((d+1)\ell+1)^d	\\
			& 	\leq |\mathcal{S}| 2^{\deg(Z')} ((d+1)\ell+1)^d	\\
			&	\leq |Z|2^{\Delta} 2^{2^{\Delta}} ((d+1)\ell+1)^d \\
			&	\leq C(\Delta) |Z|\ell^d. \qedhere
\end{align*}
\end{proof}

\textbf{Step 2: defining the simplicial map $s: Z'\to H^{d+1}(T)$}

We next define a simplicial map $s: Z'\to H^{d+1}(T)$. The image of this map will be contained in a subset $H_{\ell}$ which we now define. Fix a vertex $v\in T$ with $h(v)=-\ell$. Let $V_\ell$ be the set of vertices $w$ such that $d(v,w)\leq \ell(d+1)$ and $-\ell\leq h(w)\leq \ell d$. $V_\ell$ is the vertex set of a binary tree of depth $\ell(d+1)$ contained in $T$, so we may represent its vertices by binary strings of length at most $\ell(d+1)$. Now consider the full subcomplex of $H^{d+1}_T$ whose set of $0$-simplices is $V_\ell^{d+1}\cap (H^{d+1}_T)^0$ and denote it by $H_\ell$. We represent elements of $H_\ell$ by $(d+1)$-tuples of binary strings whose lengths sum to $\ell(d+1)$. Note that two $(d+1)$-tuples $(\omega_1,\ldots,\omega_{d+1})$ and $(\omega'_1,\ldots,\omega'_{d+1})$ span a $1$-simplex if and only if there exist $j$ and $k$ such that $\omega_j\in\{\omega'_j0,\omega'_j1\}$ and $\omega'_k\in\{\omega_k0,\omega_k1\}$; and $\omega_i=\omega'_i$ for all $i\neq j,k$.

Given this description of $H_{\ell}$, we define $s$ as follows:

Let $f\in D_{\ell}(\mathcal S)$, so $\sum_{z\in\mathcal S} f(z)=\ell(d+1)$. Denote by $b(k)_i$ the initial subword of length $i$ of the concatention of $d+1$ copies of the binary representation of $k$ as a string of length $\ell$. For example, if $\ell=7$ and $d=3$, then $b(5)_{18}=0000101\ 0000101\ 0000$.

We now define $s$. Note that $\supp(f)$ contains at most one set of any given cardinality between $1$ and $d+1$. If $\supp(f)$ contains no element of cardinality $j$, we set $s(f)_j=\emptyset$. If $A_j\in\supp(f)$ has cardinality $j$ then we define 
\[
s(f)_j = b(\min A_j)_{f(A_j)}.
\]
Note that $s(f)$ is a $(d+1)$-tuple of binary strings whose lengths sum to
\[
 \sum_{j=1}^{d+1} f(A_j) = \sum_{z\in\mathcal S} f(z) = \ell(d+1).
\]
Hence $s(f)\in H_\ell$. 
\medskip

Since $H^{d+1}_T$ is a flag complex, to prove that $s$ is simplicial it suffices to prove that whenever $f_1,f_2$ span a $1$-simplex in $Z'$, $s(f_1)$ and $s(f_2)$ span a $1$-simplex in $H^{d+1}_T$.

We write $\card(f)=\setcon{|A|}{A\in\supp(f)}$.  We introduce the shorthand
\[
 f(j) = \left\{ \begin{array}{lll} f(A) & \textup{if} & j\in \card(f),\ A\in\supp(f),\ |A|=j \\ 0 & \textup{if} & j\not\in\card(f).
 \end{array}\right.
\]
This is well-defined, because $\supp(f)$ contains at most one set of any given cardinality.

By definition, $\supp(f_1)\cup\supp(f_2) \in\mathbb S$, so it contains at most one set of any given cardinality. Moreover, there exist $j_1,j_2$ such that $f_1(j_1)=f_2(j_1)+1$ and $f_2(j_2)=f_1(j_2)+1$, and for all $j\not\in\{j_1,j_2\}$, $f_1(j)=f_2(j)$.

Firstly suppose $j\not\in\{j_1,j_2\}$. If $j\not\in\card(f_1)$ then $s(f_1)_j=\emptyset=s(f_2)_j$. If $j\in\card(f_1)$, then $s(f_1)_j= b(\min A_j)_{f_1(A_j)}$ where $A_j$ is the unique set in $\supp(f_1)\cup\supp(f_2)$ of cardinality $j$. Since $f_1(j)=f_2(j)$, we deduce that $s(f_2)_j= b(\min A_j)_{f_2(A_j)}$, so $s(f_1)_j=s(f_2)_j$.

It is clear that $j_1\in\card(f_1)$, so $s(f_1)_{j_1}=b(\min A_{j_1})_{f_1(j_1)}$. 

If $j_1\in\card(f_2)$, then $s(f_2)_{j_1}=b(\min A_{j_1})_{f_2(j_1)}$, and $f_1(j_1)=f_2(j_1)+1$, thus $s(f_1)_{j_1}\in\{s(f_2)_{j_1}0,s(f_2)_{j_1}1\}$. 

If $j_1\not\in\card(f_2)$, then $s(f_2)_{j_1}=\emptyset$ and $f_1(j_1)=1$, so $s(f_1)_{j_1}\in\{0,1\}=\{s(f_2)_{j_1}0,s(f_2)_{j_1}1\}$. The same argument holds for $j_2$. Thus $s(f_1)$ and $s(f_2)$ span a $1$-simplex in $H_{\ell}$.
\medskip

\textbf{Step 3: the coarse construction and completing the proof}

The coarse construction $\phi:Z\to \overline{H_\ell}$ is given as a concatenation of three steps $Z\to Z'' \to Z' \to \overline{H_\ell}$. The first two are the natural homeomorphisms associated to geometric subdivisions. The third map is the natural continuous extension of a simplicial map to the topological realisations. 
%
%For the second, we build the continuous map inductively on dimension. The $0$-simplices of $Z'$ are precisely those points in $Z'\subset [0,1]^{\mathcal S}$ whose support is a singleton $\{x\}$. We map this $0$-simplex to the $0$-simplex in $Z''\subset [0,1]^{D_\ell(\mathcal S)}$ which takes value $1$ on the unique $f$ with support equal to $x$, and $0$ elsewhere.
%
%We then extend this to higher simplices by mapping the closed $k$-simplex (points in $Z'\subset [0,1]^{\mathcal S}$ whose support is contained in a given $(k+1)$-tuple $\{x_0,\ldots,x_k\}$, homeomorphically onto the union of closed $k$-simplices $Z''\subset [0,1]^{D_\ell(\mathcal S)}$ whose support is contained in the collection of functions whose support is contained in the $(k+1)$-tuple.

To complete the proof, we show that the map $f:Z\to \overline{H_\ell}$ is a $2^{\deg(S)}$ coarse construction for some $k$. It is immediate that it is continuous and $f(Z^{\leq d})\subseteq \overline{H_\ell}^{\leq d}$. Let $\sigma$ be a closed simplex in $\overline{H_\ell}$, let $(\omega_0,\ldots,\omega_{d})$ be any $0$-simplex in $\sigma$ and suppose there is some $A\in\CZ$ whose image under $f$ contains $\sigma$. As the sum of the lengths of the $\omega_i$ equals $\ell(d+1)$, at least one of them has length at least $\ell$, say it is $\sigma_j$. By construction, the initial subword of length $\ell$ of $\omega_j$ is the binary sequence corresponding to the minimal element $n_j$ of $S$ contained in the simplex $A$. The number of simplices in $S$ containing $n_j$ is at most $2^{\deg{S}}$, and $\sigma'$ must be one of these.

The volume of this coarse construction is at most
\[
 |Z'|\leq C|Z|\ell^d \leq C|Z|(\log_2(|Z|)+1)^d \leq 2^dC|Z|(\log_2(1+|Z|))^d
\]
by Lemma \ref{lem:boundZ'}.
\end{proof}

Combining this result with Theorem \ref{thm:TOlbcon}

\begin{corollary} For each $q\leq n$, we have 
\[
r/\ln(1+r)^q\lesssim \sTO^q_{H^{n+1}_T}(r)\lesssim r/\ln(1+r)
\]
and for each $q<n$ we have
\[
r/\ln(1+r)^{q}\lesssim \sTO^q_{T^n}(r) \lesssim r/\ln(1+r).
\]
\end{corollary}

\subsection{Obstructing regular maps}
In this final section we prove Theorem \ref{thm:regobst}. We recall the statement for convenience.

\begin{theorem*} There is no regular map from the $0$-skeleton of $H^{d+1}$ to any product $H\times (\HH^2)^{d-1}\times D$ where $H$ is a bounded degree hyperbolic graph and $D$ is a bounded degree graph satisfying the doubling property.
\end{theorem*}
\begin{proof}
Suppose for a contradiction that there is a regular map $(H^{d+1})^0\to H\times (\HH^2)^{d-1}\times D$.

Applying the Bonk-Schramm embedding theorem to $H$ and Assouad's embedding theorem to $D$ \cite{BonkSchramm,Assouad}, we obtain a regular map
\[
 (H^{d+1})^0\to \HH^k\times (\HH^2)^{d-1}\times \R^l
\]
for some $k,l\in\N$. As $Y=\HH^k\times (\HH^2)^{d-1}\times \R^l$ is uniformly $m$-connected for every $m$, by Theorem \ref{thm:TOmono}, $\TO^d_{H^{d+1}}(r)\lesssim \TO^d_Y(r)$.

However, $\TO^d_{H^{d+1}}(r)\gtrsim r/\log(1+r)^q$ while
\begin{align*}
	\TO^d_Y(r) 	& \lesssim \TO^1_{\HH^k\times \R^l}(r) \left(\TO^1_{\HH^2}(r)\right)^{d-1} \\
				& \lesssim r^a\log(1+r)^{d-1}
\end{align*}
for some $a<1$. This is a contradiction.
\end{proof}

\newcommand{\etalchar}[1]{$^{#1}$}
\def\cprime{$'$}

\end{document}